\documentclass{amsart}

\usepackage{amsmath, amssymb, amsthm, epsfig}
\usepackage[enableskew]{youngtab}
\usepackage{hyperref}

\newcommand{\sg}{\sigma}

\newcommand{\fig}[2]{\begin{figure}[ht]
\centerline{\scalebox{.66}{\epsfig{file=#1.eps}}}
\caption{#2}
\label{fig:#1}
\end{figure}}

\def\des{\operatorname{des}}
\def\Des{\operatorname{Des}}
\def\asc{\operatorname{asc}}
\def\Asc{\operatorname{Asc}}
\def\inv{\operatorname{inv}}
\def\Inv{\operatorname{Inv}}
\def\maj{\operatorname{maj}}

\def\Obj{\operatorname{Obj}}
\def\stat{\operatorname{stat}}
\def\coinv{\operatorname{coinv}}
\def\comaj{\operatorname{comaj}}
\def\dist{D}

\def\rlmaj{\operatorname{rlmaj}}
\def\rlcomaj{\operatorname{rlcomaj}}
\def\id{\operatorname{id}}
\newcommand{\qbinom}[2] { {#1 \brack #2}_{q} }
\newcommand{\rmaj}[1] {#1 \hspace{-0.02in} \operatorname{-maj}}

\newcommand{\stir}[2] {S_{#1, #2}}
\newcommand{\qstir}[3] {S_{#1, #2}(#3)}
\newcommand{\setp}[2] {\mathcal{P}_{#1, #2}}
\newcommand{\osp}[2] {\mathcal{OP}_{#1, #2}}
\newcommand{\ospn}[1] {\mathcal{OP}_{#1}}
\newcommand{\eul}[2] {A_{#1, #2}}
\newcommand{\unc}{\operatorname{unc}}
\newcommand{\unca}{\operatorname{unca}}
\newcommand{\uncb}{\operatorname{uncb}}
\newcommand{\can}{\operatorname{can}}
\newcommand{\op}[1] {#1}
\def\S{\mathfrak{S}}
\def\rook{\mathcal{M}}

\def\multiset#1#2{\ensuremath{\left(\kern-.3em\left(\genfrac{}{}{0pt}{}{#1}{#2}\right)\kern-.3em\right)}}
\newcommand{\insertion}[3] {^{#1}_{#2, #3}}
\newcommand{\qmah}[3] {M_{#1, #2}(#3)}
\def\Falls{>^{\prime}}

\newtheorem{thm}{Theorem}[subsection]
\newtheorem{cor}{Corollary}[subsection]
\newtheorem{lemma}{Lemma}[subsection]
\newtheorem{prop}{Proposition}[subsection]

\title[Equidistribution on Ordered Set Partitions]{An Extension of MacMahon's Equidistribution Theorem to Ordered Set Partitions}
\author{Jeffrey B. Remmel and Andrew Timothy Wilson}
\address{Department of Mathematics \newline \indent
University of California, San Diego \newline \indent
La Jolla, CA, 92093-0112, USA}
\email{jremmel@math.ucsd.edu, atwilson@math.ucsd.edu (corresponding author)}
\thanks{The second author is partially supported by the Department of Defense (DoD) through the National Defense Science \& Engineering Graduate Fellowship (NDSEG) Program.}

\begin{document}
\begin{abstract}
We prove a conjecture of Haglund which can be seen as an extension of the equidistribution of the inversion number and the major index over permutations to ordered set partitions. Haglund's conjecture implicitly defines two statistics on ordered set partitions and states that they are equidistributed. The implied inversion statistic is equivalent to a statistic on ordered set partitions studied by Steingr\'{i}msson, Ishikawa, Kasraoui, and Zeng, and is known to have a nice distribution in terms of $q$-Stirling numbers. The resulting major index exhibits a combinatorial relationship between $q$-Stirling numbers and the Euler-Mahonian distribution on the symmetric group, solving a problem posed by Steingr\'{i}msson. 
\end{abstract}

\maketitle

\tableofcontents
\section{Introduction}

Let $\S_{n}$ denote the symmetric group, i.e.\ the group of permutations of 
$\{1, 2, \ldots, n\}$ under composition. Given a 
permutation $\sigma = \sigma_1 \ldots \sigma_n \in \S_{n}$, 
we define the \emph{descent} and \emph{ascent sets} of $\sigma$ to be 
\begin{align*}
\Des(\sigma) &= \{i \in \{1, 2, \ldots, n - 1\} : \sigma_{i} > \sigma_{i+1} \} 
\ \mbox{and} \\
\Asc(\sigma) &= \{i \in \{1, 2, \ldots, n - 1\} : \sigma_{i} < \sigma_{i+1} \} .
\end{align*}
The set of \emph{inversions} of $\sigma$, $\Inv(\sigma)$, is defined by 
\begin{align*}
\Inv(\sigma) &= \{(i, j) : 1 \leq i < j \leq n, ~ \sigma_{i} > \sigma_{j} \} .
\end{align*}
Then 
\begin{equation*}
\Inv^{i, \Box} = \{(i, j): i < j \leq n, ~ \sigma_{i} > \sigma_{j} \}
\end{equation*}
is the set of inversions that start  at position $i$ and 
 \begin{equation*}
\Inv^{\Box, j} = \{(i, j): 1 \leq i < j ,  \sigma_{i} > \sigma_{j} \}
\end{equation*}
is the set of inversions that end at position $j$. 
We let 
\begin{align*}
\des(\sigma) &= |\Des(\sigma)| \ \ \ \ &\inv(\sigma) &= |\Inv(\sigma)| \\
\asc(\sigma) &= |\Asc(\sigma)| \ \ \ \ &\inv^{i, \Box}(\sigma) &= |\Inv^{i, \Box}(\sigma)| \\
\maj(\sigma) &= \sum_{i \in \Des(\sigma)} i \ \ \ \ &\inv^{\Box, j}(\sigma) &= |\Inv^{\Box, j}(\sigma)|.
\end{align*}
$\des(\sg)$, $\asc(\sg)$, $\maj(\sg)$, and $\inv(\sg)$ are known as the \emph{descent number}, \emph{ascent number}, \emph{major index}, and \emph{inversion number} of $\sigma$, respectively.

This paper was motivated by the following conjecture of Jim Haglund: 
\begin{equation}
\label{haglund-conj}
\sum_{\sigma \in \S_{n}} q^{\inv(\sigma)} \prod_{i \in \Des(\sigma)} \left(1 + \frac{z}{q^{1 + \inv^{\Box, i}(\sigma)}} \right) = \sum_{\sigma \in \S_{n}} q^{\maj(\sigma)} \prod_{j = 1}^{\des(\sigma)} \left(1 + \frac{z}{q^{j}} \right).
\end{equation}
We will give a bijective proof of (\ref{haglund-conj}) by viewing 
it as saying that a certain pair of statistics defined on ordered 
set partitions of $\{1, \ldots, n\}$ are equidistributed. In general, 
a statistic on a set of objects $\Obj$ is a map 
$\stat$ from $O$ into the set of natural numbers $\mathbb{N} = \{0,1,2, \ldots \}$. Given a sequence of statistics 
$\stat_{1}, \stat_{2}, \ldots, \stat_{m}$ on $\Obj$, the \emph{distribution} of $(\stat_{1}, \ldots, \stat_{m})$ over $\Obj$ is the polynomial
\begin{align*}
\dist^{\stat_{1}, \ldots, \stat_{m}}_{\Obj}(x_{1}, \ldots, x_{m}) &= \sum_{\sigma \in \Obj} \prod_{i=1}^{m} x_{i}^{\stat_{i}(\sigma)} .
\end{align*}
Then two sequences of statistics $(\stat_{1}, \ldots, \stat_{m})$ on $\Obj$ and $(\stat^{\prime}_{1}, \ldots, \stat^{\prime}_{m})$ on $\Obj^{\prime}$, are said to be \emph{equidistributed} if
\begin{displaymath}
\dist^{\stat_{1}, \ldots, \stat_{m}}_{\Obj}(x_{1}, \ldots, x_{m}) = \dist^{\stat^{\prime}_{1}, \ldots, \stat^{\prime}_{m}}_{\Obj^{\prime}}(x_{1}, \ldots, x_{m}).
\end{displaymath}

Permutation statistics have long played a fundamental role in combinatorics. For example, consider the usual $q$-analogue of $n!$
\begin{align*}
[n]_{q}! &= [1]_{q}[2]_{q}\ldots[n-1]_{q}[n]_{q}
\end{align*}
where $[n]_{q} = 1 + q + \cdots + q^{n-1}$. 
In \cite{macmahon}, originally published in 1915, MacMahon showed that inversion number and major index are equidistributed over $\S_{n}$, and that
\begin{align}
\label{mahonian}
\dist^{\inv}_{\S_{n}}(q) = \dist^{\maj}_{\S_{n}}(q) = [n]_{q}!.
\end{align}
In his honor, any permutation statistic with this distribution over $\S_{n}$ is said to be \emph{Mahonian}.
The equidistribution of inversion number and major index was proved bijectively for the first time by Foata \cite{foata}. Carlitz gave another 
bijection in \cite{carlitz}. 

Clearly setting $z=0$ in (\ref{haglund-conj}) gives (\ref{mahonian}),
so Haglund's conjecture is an extension of (\ref{mahonian}). We shall 
show that  (\ref{haglund-conj}) can be viewed as a statement that two statistics 
$\inv$ and $\maj$ 
on the set of ordered set partitions are equidistributed.  It turns out that the  $\inv$ statistic 
is equivalent to a statistic on ordered set partitions 
studied by Steingr\'{i}msson  \cite{stein}. It follows 
from the work of  Steingr\'{i}msson that the coefficient 
of $z^k$ on the left-hand side of (\ref{haglund-conj}) is 
$[n-k]_q!S_{n,n-k}(q)$ where $S_{n,k}(q)$ is a $q$-analogue 
of the Stirling number of the second kind $S_{n,k}$ which is defined 
by the recursions 
$$S_{n+1,k}(q) = S_{n,k-1}(q)+[k]_qS_{n,k}(q)$$
with initial conditions $S_{0,0}(q) =1$ and $S_{n,k}(q)= 0$ if 
$k < 0$ or $n < k$. The statistic $\maj$ is related to the statistic bmajmil in \cite{stein}, although we will see that our different perspective is quite valuable. Furthermore, we will show that our bijective 
proof of 
(\ref{haglund-conj}) allows us to give to give a bijective 
proof of a combinatorial relationship between the $q$-Stirling numbers and 
a certain distribution on the symmetric group which solves a 
problem posed by Steingr\'{i}msson.

The outline of this paper is as follows.  In Section \ref{sec:insertion}, 
we review Carlitz's insertion method to prove (\ref{mahonian}). 
In particular, we state the key labeling lemmas 
for $\inv$ and $\maj$ for permutations. 
In Section \ref{sec:file}, we give rook theory interpretations 
of $[n]_q!$ and $[n-k]_q!S_{n,n-k}(q)$ and show how those interpretations   
lead to a natural interpretation of the recursions satisfied 
by  $[n-k]_q!S_{n,n-k}(q)$.  In Section \ref{sec:osp}, we 
define extensions of $\inv$ and $\maj$ to the set of 
ordered set partitions and prove analogues 
of the labeling lemmas which allows us to give a bijective 
proof of Haglund's conjecture (\ref{haglund-conj}). In Section \ref{sec:app}, 
we describe several extensions of Haglund's conjecture and 
give a bijective proof of Steingr\'{i}msson's problem.

\section{The Insertion Method}
\label{sec:insertion}
We begin by reviewing a particular bijection on permutations that maps the inversion number to the major index. This bijection is due to work by Carlitz in \cite{carlitz}, and has come to be known as the insertion method, as it involves accounting for the effects of inserting a new largest element into a permutation. 

Let $\sg$ permutation in $\S_{n-1}$. Then there
are $n$ spaces where we can insert $n$ in $\sg$ to obtain a
permutation $\tau \in \S_n$, namely, immediately before
$\sg_1$ or immediately after $\sg_i$ for $i =1, \ldots, n-1$. We are interested in how this insertion affects the inversion number and major index of the permutation. In order to keep track of these changes, we define two labelings of the $n$ spaces in  which we could insert $n$ into $\sg$.

The \emph{inv-labeling} for $\sg$ 
is the labeling obtained by numbering the spaces
from right to left  with $0,1, \ldots, n$. To get the \emph{maj-labeling} for $\sg$, we label the space after $\sg_{n-1}$ with
0, then label the spaces following the descents of
$\sg$ from right to left with $1,2, \ldots, \des(\sg)$, and then
label the remaining spaces from left to right with
$\des(\sg)+1, \ldots, n$. For example, if
$\sg = 14352$, we can write the $\inv$-labeling of the spaces of $\sg$ as subscripts to get
\begin{align} \label{ex:inv-label}_5  1_4 4_3 3_2 5_1 2_0. \end{align}
The $\maj$-labeling of $\sg$ is
\begin{align} \label{ex:maj-label}_3 1_4 4_2 3_5 5_1 2_0. \end{align}

These labels will work together with insertion maps to build permutations in $\S_{n}$. For $n \geq 2$, we define the maps
\begin{align*}
\phi_{\inv, n},  \phi_{\maj, n} : \{0, 1, \ldots, n-1\} \times \S_{n-1} \to \S_{n}
\end{align*} by sending $(i, \sg)$ to the permutation obtained by inserting $n$ in the position labeled $i$ in the $\inv$-labeling (respectively $\maj$-labeling) of $\sg$. For example,
\begin{equation*}
\phi_{\inv, 6}(2, 14352) = 143652 \ \mbox{and} \ 
\phi_{\maj, 6}(2, 14352) = 146352. 
\end{equation*}
Then we have the following two lemmas, which we will call insertion lemmas.
\begin{lemma}\label{lemma:inv}
If $\sg \in \S_{n-1}$, then
$\inv(\phi_{\inv, n}(i, \sg)) = \inv(\sg) +i$.
\end{lemma}

\begin{lemma}\label{lemma:maj}
If $\sg  \in \S_{n-1}$, then
$\maj(\phi_{\inv, n}(i, \sg)) = \maj(\sg) +i$.
\end{lemma}

Lemma \ref{lemma:inv} is straightforward to prove and
Lemma \ref{lemma:maj} is essentially due to Carlitz. For a detailed proof of a generalization of Lemma \ref{lemma:maj}, see \cite{hlr}. We can use these lemmas to prove MacMahon's Theorem. It is easy to see that 
$\phi_{\inv, n}$ and $\phi_{\maj, n}$ map  $\{1, \ldots, n\} 
\times \S_{n-1}$ onto $\S_n$. Hence 
\begin{align*}
\dist^{\inv}_{\S_{n}}(q) &= \sum_{\tau \in \S_{n}} q^{\inv(\tau)} = \sum_{i=0}^{n-1} \sum_{\sg \in \S_{n-1}} q^{\inv(\phi_{\inv, n}(i, \sg))} \\
&= \sum_{i=0}^{n-1} \sum_{\sg \in \S_{n-1}} q^{\inv(\sg) + i} = [n]_{q} \dist^{\inv}_{\S_{n-1}}(q) = [n]_{q}!
\end{align*}
by induction. The same computation holds for the major index.

This approach also yields a recursive bijection that shows 
that the inversion number and the major index are equidistributed. We define $\psi_{1} : \S_{1} \rightarrow \S_{1}$ to be the identity map and recursively set $\psi_{n} : \S_{n} \rightarrow \S_{n}$ as
\begin{align*}
\psi_{n} = \phi_{\maj, n} \circ (\id, \psi_{n-1}) \circ (\phi_{\inv, n})^{-1}
\end{align*}
for $n \geq 2$, where $\id$ is the identity map. Since $\psi_{n}$ is a composition of bijections, it is also a bijection. Furthermore, Lemmas \ref{lemma:inv} and \ref{lemma:maj} prove that, for any $\tau \in \S_{n}$, $\maj(\psi_{n}(\tau)) = \inv(\tau)$. To see this, write $(\phi_{\inv, n})^{-1}(\tau)$ as $(i, \sigma)$. Then
\begin{align*}
\inv(\tau) &= \inv(\phi_{\inv, n}(i, \sg)) = i + \inv(\sg) \\
&= i + \maj(\psi_{n-1}(\sg)) = \maj( \phi_{\maj, n}(i, \psi_{n-1}(\sg)) ) = \maj(\psi_{n}(\tau))
\end{align*}
by induction.

To compute $\psi_{5}(52143)$, we first compute $(\phi_{\inv, 5})^{-1}(52143)$ by removing $5$ and counting the number of inversions lost by removing 5. In this case, we have lost 4 inversions. We record this number in the $i$ column and the resulting permutation in the $\tau$ column. We repeat this process until we have reached $n = 1$ and filled the first three columns of the table. To build our new permutation, we recursively place $n$ at the position that receives label $i$ in the major index labeling. This process is pictured below. 

\begin{align*}
\begin{array}{l l l l}
n & \tau & i & \psi_{n}(\tau) \\\hline
5 & 52143 & - & 24153 \\
4 & 2143 & 4 & _{2}2_{3}4_{1}1_{\emph{4}}3_{0} \\
3 & 213 & 1 & _{2}2_{\emph{1}}1_{3}3_{0} \\
2 & 21 & 0 & _{2}2_{1}1_{\emph{0}} \\
1 & 1 & 1 & _{\emph{1}}1_{0}
\end{array}
\end{align*}


\section{Rook Theory Interpretations}
\label{sec:file}
In this section, we give rook theory interpretations 
of $[n]_q!$ and $[n-k]_q!S_{n,n-k}(q)$. As we shall see, this point 
of view will be helpful in understanding our extension of 
the insertion method. 

Let $F(b_1, \ldots, b_n)$ be the rook board that has 
$b_i$ cells in the $i$th column for $i =1, 
\ldots, n$. Given a board $B = F(b_1, \ldots, b_n)$, 
let $\mathcal{F}_k(B)$ denote the set of all placements 
of $k$ rooks in $B$ such that there is at most one rook in each 
column.  (In this setting, rooks may share rows.) We will call an element $F \in \mathcal{F}_k(B)$ a 
\emph{file placement} of $k$ rooks in $B$. Let $\mathcal{N}_k(B)$ denote the set of all placements 
of $k$ rooks in $B$ such that there is at most one rook in each row and 
column.  We will call an element $P \in \mathcal{N}_k(B)$ 
a \emph{non-attacking rook placement} of $k$ rooks in $B$. 

Next, we introduce a statistic on these placements.
If $F \in \mathcal{F}_{k}(B)$, we will think of each rook $r$ in 
$F$ as canceling all the cells in its column that lie above 
$r$ plus the cell contains $r$.  Then we let 
$\unc_B(F)$ denote the number of uncanceled cells for $F$ that 
lie below some rook $r \in F$. Similarly, if 
$P \in \mathcal{N}_{k}(B)$, we will think of each rook $r$ in 
$P$ as canceling all the cells in its column that lie above 
$r$ and all the cells that lie in its row to the right of 
$P$ plus the cell that contains $r$.  Then we let 
$\unc_B(P)$ denote the number of uncanceled cells for $P$ that 
lie below some rook $r \in P$. For example, 
in Figure \ref{fig:filerook}, we have pictured a file 
placement $F$ in $B_7 =F(0,1,2,3,4,5,6)$ on the left where we have indicated 
the rooks with circled $X$s and placed dots in the cells 
which do not contain rooks that are canceled by a rook in $F$. 
Similarly, we have pictured a non-attaching 
rook placement $P$ on the right  where we have indicated 
the rooks with $X$s and placed dots in the cells 
which do not contain rooks that are canceled by a rook in $P$.  
In this case, $\unc_{B_7}(F) =3$ because there are three uncanceled cells that 
lie below a rook in $F$ and $\unc_{B_7}(P) =4$ since there are 
four uncanceled cells that lie below a rook in $P$. 

\fig{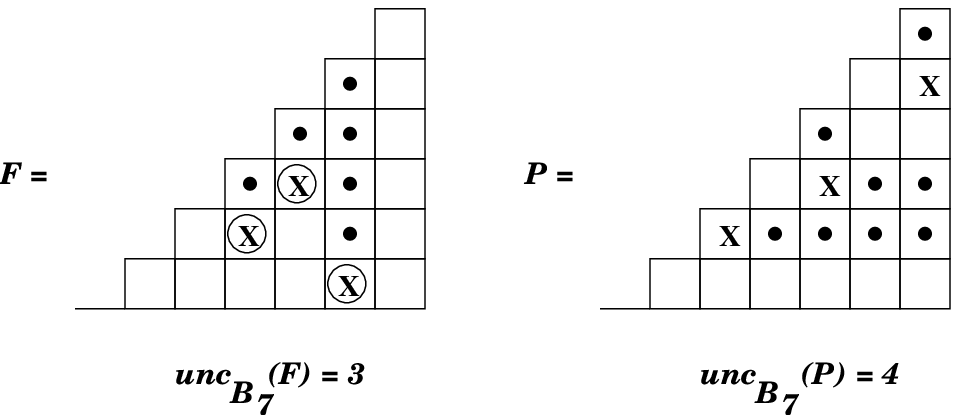}{A file placement $F$ and a rook placement $P$ in $B_7$.}

We let $B_n = F(0,1, \ldots, n-1)$ and $St_n = F(1,2,\ldots,n)$ be 
the staircase boards that start with 0 and 1, respectively. 
Let $\mathcal{F}_n = \mathcal{F}_n(St_n)$.  Thus file 
placements of $\mathcal{F}_n$ must have one rook in each column. For 
example, Figure
\ref{fig:file} pictures a file placement $F$ in $\mathcal{F}_6$ 
with 7 uncanceled cells.

\fig{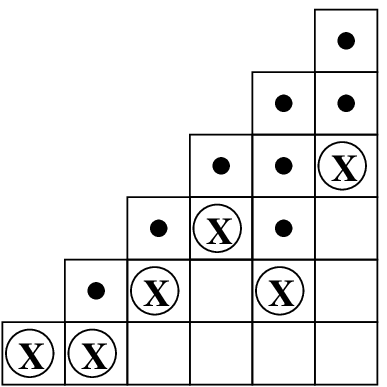}{A file placement in $St_7$.}

Note the contribution of 
the rook $r_i$ in the $i$th column to 
$\sum_{F \in \mathcal{F}_n} q^{\unc_{St_n}(F)}$ is clearly, 
$1+q+ \cdots + q^{i-1}=[i]_q$ since placing $r_i$ in the $j$th row 
give $j-1$ uncanceled cells. Thus 
$\dist^{\unc}_{\mathcal{F}_{n}}(q) = [n]_q!$. In other words, the statistic $\unc$ is Mahonian over $\mathcal{F}_n$.

We can set
\begin{align*}
\phi_{\unc, n}: &\{0, 1, \ldots, n-1\} \times \mathcal{F}_{n-1} \rightarrow \mathcal{F}_{n}
\end{align*}
by mapping $(i, F)$ to the rook placement obtained by adding a new column to the right-hand side of $F$ and placing a rook in row $i+1$ in that column. Then we have $\unc(\phi_{\unc, n}(i, F)) = i + \unc(F)$. This allows us to recursively build maps between file placements and permutations that send the statistic $\unc$ to inversion number and major index. That is, we say that $\alpha_{1}$ and $\beta_{1}$ map $\mathcal{F}_{1}$ to $\S_{1}$ in the obvious manner and recursively define
\begin{align*}
\alpha_{n} &= \phi_{\inv, n} \circ (\id, \alpha_{n-1}) \circ (\phi_{\unc, n})^{-1} \\
\beta_{n} &= \phi_{\maj, n} \circ (\id, \beta_{n-1}) \circ (\phi_{\unc, n})^{-1} .
\end{align*}
An example of the construction 
of $\alpha_5(F)$ and $\beta_5(F)$ for an $F \in \mathcal{F}_5$ 
is given in Figure \ref{fig:alphabeta}. 

\fig{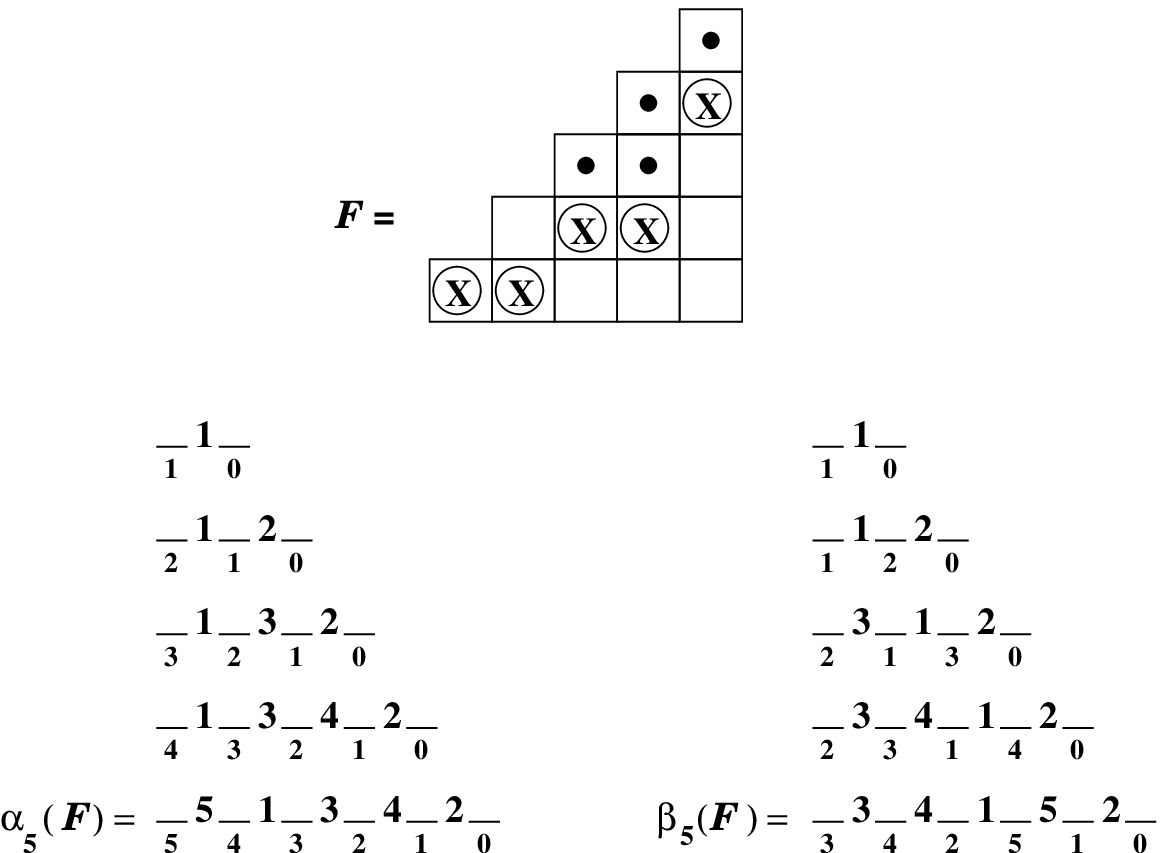}{An an example of $\alpha_5(F)$ and $\beta_5(F)$.}

It follows that
\begin{align*}
\inv(\alpha_{n}(F)) &= \unc(F) = \maj(\beta_{n}(F))
\end{align*}
so $\alpha$ sends $\unc$ to $\inv$ and $\beta$ sends $\unc$ to $\maj$. We also note that $\beta_{n} \circ (\alpha_{n})^{-1} = \psi_{n}$; in other words,
one can use file placements in $\mathcal{F}_n$ to define 
the map  $\psi_{n}$.

In \cite{gr-rook}, the authors gave a combinatorial interpretation of the polynomial $\qstir{n}{k}{q}$.  That is, $\qstir{n}{k}{q} = 
\sum_{P \in \mathcal{N}_{n-k}(B_n)} q^{\unc_{B_n}(P)}$. 
Indeed, the recursion
$S_{n+1,k}(q) = S_{n,k-1}(q)+[k]_qS_{n,k}(q)$ classifies 
the rook placements $P \in \mathcal{N}_{n+1-k}(B_{n+1})$ by whether 
or not there is a rook in the last column.  
If we have no rook in the last column, then we have 
$n+1-k$ non-attacking rooks in the first $n$ columns which contribute 
$S_{n,k-1}(q)$ to $\sum_{P \in \mathcal{N}_{n+1-k}(B_{n+1})} 
q^{\unc_{B_{n+1}}(P)}$. 
If there is a rook in the last column, then there are 
$n-k$ non-attacking rooks in the first $n-1$ columns and 
these will cancel $n-k$ cells in the last column which is 
of height $n$. Thus we have $k$ cells in which we can place 
the rook in the last column and we get an extra weight of 
$q^{i-1}$ if we place the rook in the $i$th available cell from 
the bottom. Thus the placements $P \in \mathcal{N}_{n-k}(B_n)$ with 
a rook in the last column contribute $(1+q+ \cdots +q^{k-1})S_{n,k}(q) = 
[k]_qS_{n,k}(q)$ to 
$\sum_{P \in \mathcal{N}_{n+1-k}(B_{n+1})} q^{\unc_{B_{n+1}}(P)}$. 

To give a combinatorial interpretation to 
$[n-k]_q!S_{n,n-k}(q)$, we introduce \emph{mixed 
placements}, which contain both file rooks and non-attacking rooks. 
To our knowledge, these have not received any attention in the literature on rook theory.
Given a board, we wish to place both file rooks and non-attacking rooks in the board. In particular, we insist that there is at most one rook in each column and no rook lies in a cell which is canceled by a rook to its left. If one thinks of starting with a placement of the non-attacking rooks and then ``completes'' this placement by placing the file rooks, avoiding canceled cells, one obtains a mixed placement. An example of this process is pictured in Figure \ref{fig:mixed1}.

\fig{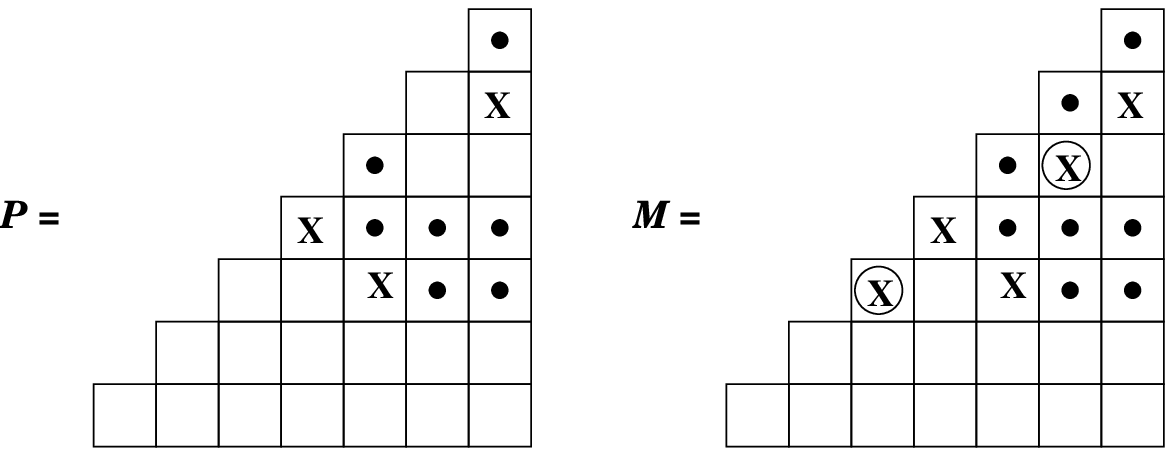}{A mixed  placement.}

Formally, we let
$\mathcal{M}_{n,k}$ denote the set of all placements
$P = N \cup F$ where $N$ is a placement of
$k$ non-attacking rooks in $B_n$ and $F$ is a file
placement of $n-k$ rooks in $B_n$ such that
\begin{itemize}
\item no rook in $N$ is in the first row\footnote{We explore what happens when we remove this condition in Section \ref{ssec:beyond}.},
\item there is one rook in each column,
\item each rook $r$ in $N$ cancels the cell it occupies,
all cells in its row that lie to right of $r$, and all
cells in its column that either lie above $r$ or lie in the first
row,
\item each rook $f$ in $F$ cancels the cell it occupies plus all
cells in its column that lie above $f$, and
\item no rook lies in a cell which is canceled by another
rook.
\end{itemize}
Given a placement $P = N \cup F \in \mathcal{M}_{n,k}$, we
let $\unc(P)$ equal the number of uncanceled cells in $P$.
We call the placements $P = N \cup F \in \mathcal{M}_{n,k}$ \emph{mixed rook 
placements} and refer to rooks in $N$ as non-attacking rooks and
the rooks in $F$ as file rooks.
For example, in Figure \ref{fig:mixed}, we have pictured an
element of $P = N \cup F \in \mathcal{M}_{7,3}$ where the
rooks in  $N$ are denoted by $X$s, the elements of
$F$ are denoted by circled $X$s, and the canceled cells are indicated
by placing a dot in them. In this case, $\unc(P) = 9$.

\fig{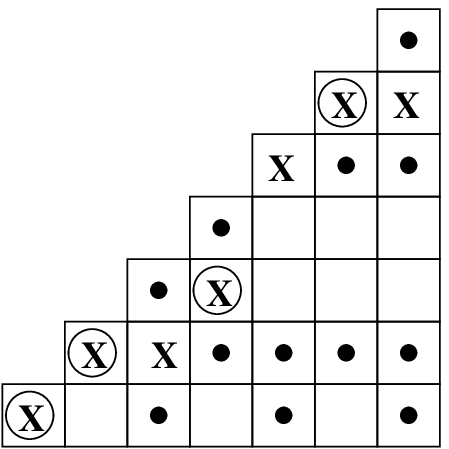}{A placement in $\mathcal{M}_{7,3}$.}

We claim that the distribution of uncanceled cells over $\mathcal{M}_{n,k}$ is equal to $[n-k]_{q}! S_{n,n-k}(q)$. To see this, we consider how we can build mixed placements in $\mathcal{M}_{n,k}$. We start with a rook placement
$P \in \mathcal{N}_{k}(B_{n})$. Then we add a row of $n$ cells at the bottom of $B_{n}$
to obtain the board $St_n$.  Each rook 
in $P$ also cancels the cells in this extra row which lies below it. 

At this point, we claim that the number
of uncanceled cells in the columns that do not contain rooks
in $P$ are $1, \ldots, n-k$ as we read from left to right. We prove this by induction on the number of rooks $k$. Clearly, the base case
$k=0$ automatically holds.
Now suppose that our claim is true for all placements of
$k-1$ non-attacking rooks in $B_{n}$. Consider a placement $Q$
of $k$ non-attacking rooks in $B_{n}$. Let $j$ be the column which
contains the rightmost rook in $Q$ and let $P$ be the rook placement
that results from $Q$ by removing the rook in column $j$.
By induction, there is some $\ell$ such that there are
$\ell$ uncanceled cells in column $j$ relative to $P$ and, hence,
the number of uncanceled cells in empty columns to the left
of column $j$ relative to $P$
as we read from left to right are $1,2, \ldots \ell-1$ and the
number of uncanceled cells in columns $j,j+1, \ldots, n$ are
$\ell,\ell+1, \ldots, n-(k-1)$, respectively.  The effect of putting a rook
in column $j$ is to remove one uncanceled cell in each of columns
$j+1, \ldots, n$. Hence, relative to $Q$, the number of uncanceled
cells in columns $j+1, \ldots, n$ will be
$\ell,\ell+1, \ldots, n-k$, respectively, as desired.

It follows that if we consider the uncanceled cells in the empty columns
of $P \in \mathcal{N}_{k}(B_{n-1})$, we have a copy of the board $St_{n-k}$. We fill this embedded board with a file placement, keeping track of the distribution of uncanceled cells over this file placement. Therefore 
\begin{align*}
\dist^{\unc}_{\mathcal{M}_{n, k}}(q) = [n-k]_q! S_{n,n-k}(q) .
\end{align*}

Moreover, it is easy to see that 
$\dist^{\unc}_{\mathcal{M}_{n, 0}}(q) =[n]_q!$ since in that 
case, $\mathcal{M}_{n, 0}$ is just the set of file 
placements in $St_n$.  Similarly, 
$\dist^{\unc}_{\mathcal{M}_{n, n}}(q) =0$ since we can not 
place $n$ non-attacking rooks in $B_n$.  
We also have the following 
recursion when $1\leq k < n$:
\begin{equation} \label{mixedrec}
 \dist^{\unc}_{\mathcal{M}_{n, k}}(q) = [n-k]_q  
\dist^{\unc}_{\mathcal{M}_{n-1, k-1}}(q) + [n-k]_q  
\dist^{\unc}_{\mathcal{M}_{n-1, k}}(q).
\end{equation}
To prove this, we simply classify the placements in 
$\mathcal{M}_{n,k}$ by whether the rook $r_n$ in the last 
column is a file rook or a non-attacking rook. If 
$r_n$ is a file rook, then there are $k$ non-attacking rooks 
in the first $n-1$ columns of $St_n$ and they cancel $k$ 
cells in the last column which is of height $n$.  Hence 
we have $n-k$ cells in which we can place the file rook in the last 
column. In that case, if $r_n$ is placed in the $i$th available 
cell from the bottom, it will contribute a factor of $q^{i-1}  
\dist^{\unc}_{\mathcal{M}_{n-1, k}}(q)$ to $ 
\dist^{\unc}_{\mathcal{M}_{n-1, k}}(q)$. Hence the set of placements 
with a file rook in the last column contributes  
$(1+q+ \cdots +q^{n-k-1}) \dist^{\unc}_{\mathcal{M}_{n-1, k}}(q) = [n-k]_q  
\dist^{\unc}_{\mathcal{M}_{n-1, k}}(q)$ to $ 
\dist^{\unc}_{\mathcal{M}_{n-1, k}}(q)$.  Similarly, 
if there is a non-attacking rook in the last column, then 
there are $k-1$ non-attacking rooks in the first $n-1$ columns. 
Thus there are $n-k+1$ uncanceled cells in the last column, but 
we can not put a rook in the first row since it is a non-attacking 
rook so that are only $n-k$ available cells to place the non-attacking 
rook $r_n$ in the last column. Again, 
if $r_n$ is placed in the $i$th available 
cell from the bottom, it will contribute a factor of $q^{i-1}  
\dist^{\unc}_{\mathcal{M}_{n-1, k-1}}(q)$ to $ 
\dist^{\unc}_{\mathcal{M}_{n-1, k}}(q)$ . Hence the set of placements 
with a file rook in the last column contributes 
$(1+q+ \cdots +q^{n-k-1}) \dist^{\unc}_{\mathcal{M}_{n-1, k-1}}(q) = [n-k]_q  
\dist^{\unc}_{\mathcal{M}_{n-1, k-1}}(q)$ to $ 
\dist^{\unc}_{\mathcal{M}_{n-1, k}}(q)$.

Finally, we observe that elements of $\mathcal{M}_{n,k}$ can naturally 
be identified with an ordered set partition of $\{1, \ldots, n\}$ 
with $n-k$ parts\footnote{This is \emph{not} equal to either of the bijections $\Delta^{>}_{n,k}$ or $\Gamma^{>}_{n,k}$ that we develop in Section \ref{sec:osp}.}. We can think of taking 
$P \in \mathcal{M}_{n,k}$ and decomposing it into an element 
of $N \in \mathcal{N}_k(B_n)$ which comes from the non-attaching rooks in $P$ 
and an element of $F \in \mathcal{F}_{n-k}$ which is determined by 
the file rooks in $P$. For example, we have pictured 
the decomposition of the mixed placement $P$ pictured 
in Figure \ref{fig:mixed} in Figure \ref{fig:mixed2}. The classical 
way to think of an element $N \in \mathcal{N}_k(B_n)$ as a set 
partition $\pi(N)$ of $\{1, \ldots, n\}$ with $n-k$ parts 
is to label the rows with $1, \ldots n$ reading from top to bottom 
and interpreting an $X$ in cell $(i,j)$ as telling us 
that $i$ and $j$ are in the same part. For example, this process 
is pictured on the left in Figure \ref{fig:mixed2}. It is easy to see 
that minimal elements in the set partition correspond to the columns 
that do not contain rooks.  As we described in the last section, 
we can view an element $F \in \mathcal{F}_{n-k}$ as a permutation 
$\sg(F)$ in $\S_{n-k}$ using the $\inv$- (or even $\maj$-)labeling described above. 
Then we can view the pair $(\pi(N),\sg(F))$ as an ordered 
set partition of $\{1, \ldots, n\}$ with $n-k$ parts by ordering the parts 
according the permutation of the minimal elements induced 
by $\sg(F)$. This process is pictured at the bottom 
of Figure \ref{fig:mixed2}. We will develop more maps from mixed rook placements to ordered set partitions in Section \ref{sec:osp}. 

\fig{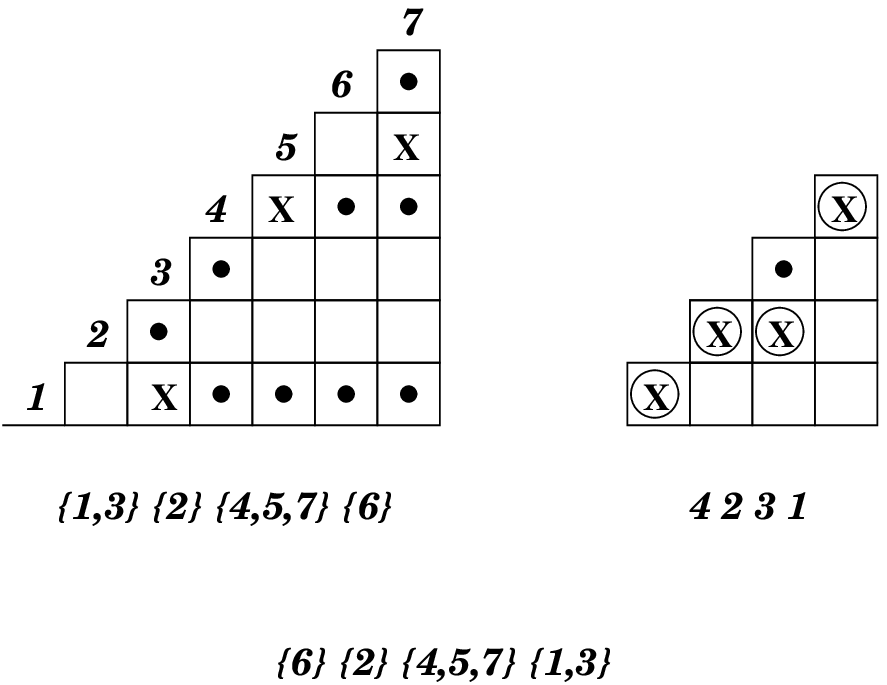}{The ordered set partition associated with the 
mixed placement in Figure 5.}

\section{An Insertion Method for Ordered Set Partitions}
\label{sec:osp}

In this section we prove our main theorem by generalizing Carlitz's insertion method to ordered set partitions.

\subsection{Statistics on Ordered Set Partitions}
\label{ssec:stats}

In our generalization we replace permutations in $\S_n$ with ordered set 
partitions of $\{1, \ldots, n\}$. 
 A \emph{set partition} of an $n$-element set is a partition of the set into nonempty subsets, called \emph{blocks}. An \emph{ordered set partition} is formed by giving an order to the blocks of a partition. For example,
 \begin{align*}
 \{\{1, 4\}, \{2, 3, 7\}, \{5\}, \{6\} \}
 \end{align*}
is a set partition of $\{1, 2, \ldots, 7\}$ with 4 blocks. We can choose to order the blocks in this partition to obtain an ordered set partition in $4! = 24$ ways, one of which results in the ordered set partition
\begin{align*}
\{2,3,7\} , \{6\}, \{1,4\}, \{5\}.
\end{align*}
Instead of using set brackets, it is common to use bars to separate each block, so this ordered set partition is written as either $237|6|14|5$ or $732|6|41|5$, depending on whether we choose to write blocks in ascending or descending order, respectively.We will denote the family of set partitions of $\{1, 2, \ldots, n\}$ into $k$ blocks by $\setp{n}{k}$ and the family of ordered set partitions of $\{1, 2, \ldots, n\}$ into $k$ blocks by $\osp{n}{k}$. We set $\ospn{n} = \bigcup_{k=1}^{n} \osp{n}{k}$. The cardinality of $\setp{n}{k}$ is the Stirling number of the second kind, which we will write $\stir{n}{k}$. Hence $|\osp{n}{k}| = k!\stir{n}{k}$ and it easily follows from the recursions for $\stir{n}{n-k}$ that 
\begin{align*}
|\osp{n}{n-k}| = (n-k) |\osp{n-1}{n-k-1} | + (n-k) | \osp{n-1}{n-k} | .
\end{align*}
One way to combinatorially prove this recursion is to notice that every ordered set partition in $\osp{n}{m-k}$ is formed by adding $n$ to some ordered set partition in $\ospn{n-1}$. In particular, we can either add $n$ as its own block to an element of $\osp{n-1}{n-k-1}$ or we can add $n$ to one of the existing blocks in an element of $\osp{n-1}{n-k}$. Each of these can be done in $n-k$ ways. 

There is another way of thinking of ordered set partitions that will be especially useful for our purposes, namely the concept of an ordered set partition as an ascent-starred or descent-starred permutation. Instead of using bars to signify separations between blocks, we can mark spaces between elements that share a block with stars. For example, $237|6|14|5$ becomes $2_*3_*7~6~1_*4~5$, or, if blocks are written in decreasing order, $732|6|41|5$ becomes $7_*3_*2~6~4_*1~5$.  Thus we have established canonical bijections between ordered set partitions $\osp{n}{k}$ and the sets
\begin{align*}
&\{(\sigma, S): \sigma \in \S_{n}, ~ S \subseteq \Asc(\sg), ~ |S| = n-k \}  \ \mbox{and} \\
&\{(\sigma, S): \sigma \in \S_{n}, ~ S \subseteq \Des(\sg), ~ |S| = n-k \}.
\end{align*}
We will refer to these as \emph{ascent-starred} and \emph{descent-starred permutations}, where the set $S$ gives the positions of the starred ascents (or descents). Our notation for these sets will be
\begin{align*}
\S^{<}_{n, k} &= \{(\sigma, S): \sigma \in \S_{n}, ~ S \subseteq \Asc(\sg), ~ |S| = k \} \ \mbox{and} \\
\S^{>}_{n, k} &= \{(\sigma, S): \sigma \in \S_{n}, ~ S \subseteq \Des(\sg), ~ |S| = k \} .
\end{align*}
For convenience we set
\begin{align*}
\S^{<}_{n} &= \bigcup_{k=0}^{n-1} \S^{<}_{n, k}\ \mbox{and} \  \S^{>}_{n} = \bigcup_{k=0}^{n-1} \S^{>}_{n, k} .
\end{align*}

In order to prove our theorem, we will interpret 
Haglund's conjecture (\ref{haglund-conj}) as a statement 
about the equidistribution of two statistics on descent-starred permutations. 

We start with the inversion side. First, we take some permutation $\sigma \in \S_{n}$ and calculate its inversion number $\inv(\sigma)$. 
Next for  
each $i \in \Des(\sg)$, we must consider the factor 
$1 + \frac{z}{q^{1 + \inv^{\Box, i}(\sigma)}}$. We will think 
of the choice of 1 in this factor as telling us to not star the 
descent $i$ and the choice of $\frac{z}{q^{1 + \inv^{\Box, i}(\sigma)}}$ 
as telling us to star the descent $i$.  
Therefore, for each starred descent at position $i$, we want to subtract
$\displaystyle 1 + \inv^{\Box, i}(\sigma)$ 
from the power of $q$. To account for the $\inv^{\Box, i}(\sigma)$ term, we will ignore all inversions that end at starred descents. We must ignore one more 
inversion for each star. Since each block is decreasing, we know that there is an inversion between any starred element and the rightmost (i.e.\ minimal) element in its block. This is the extra inversion that we will subtract. Therefore the \emph{inversions} of a descent-starred permutation $(\sg, S) \in \S^{>}_{n}$ are
\begin{align*}
\Inv((\sg, S)) &= \{(i, j): 1 \leq i < j \leq n, ~ \sg_{i} > \sg_{j}, ~ j \notin S, ~ \{i, i+1, \ldots, j-1\} \not \subseteq S \}
\end{align*}
and the \emph{inversion number} of $(\sg, S)$ is
\begin{align*}
\op{\inv}((\sg, S)) &= |\Inv((\sg, S))| = \inv(\sg) - \sum_{i \in S} 1 + \inv^{\Box, i}(\sg) .
\end{align*}
For example, if $(\sg, S) = 7_*3_*2~6~4_*1~5$,
\begin{align*}
\op{\Inv}((\sg, S)) &= \{(1, 4), (1, 6), (1, 7), (2, 6), (3, 6), (4, 6), (4, 7) \} \ \mbox{and}\\ 
\op{\inv}((\sg, S)) &= 7 \\ 
 &= \inv(\sg) -  \sum_{i \in S} 1 + \inv^{\Box, i}(\sg) \\ 
 &= \inv(\sg) - \inv^{\Box, 1}(\sg) - \inv^{\Box, 2}(\sg) - \inv^{\Box, 5}(\sg) - 3 \\ 
&= 13 - 0 - 1 - 2 - 3 = 7 .
\end{align*}
It follows that 
\begin{align*}
&\sum_{\sigma \in \S_{n}} q^{\inv(\sigma)} \prod_{i \in \Des(\sigma)} \left(1 + \frac{z}{q^{1 + \inv^{\Box, i}(\sigma)}} \right) = \sum_{k=1}^{n}  \dist^{\op{\inv}}_{\S^{>}_{n, k}}(q) z^{k} .
\end{align*}

Before moving on to the major index side, we note that we can define an inversion statistic for ordered set partitions and ascent-starred permutations by following our canonical bijections from descent-starred permutations to ascent-starred permutations. For an ordered set partition, we observe that our definition of $\op{\inv}$ counts exactly the inversions that are in different blocks where the smaller element is minimal in its block. For example, if  we think of $(\sg, S) = 7_*3_*2~6~4_*1~5$ 
as the ordered set partition $\{2,3,7\},\{6\},\{1,4\},\{5\}$ then 
$6$ contributes 1 inversion, 1 contributes 4 inversions, and 
5 contributes 2 inversions. This point of view makes it clear that our $\op{\inv}$ statistic is exactly equal to the statistic $\operatorname{ros}$ defined in \cite{stein}.

It follows that if we start with an ordered set partition 
$\pi$ of $\{1, \ldots, n\}$ which corresponds to 
$(\sg,S) \in \S_n^{>}$ and $(\tau,T) \in \S_n^{<}$, then 
the corresponding inversions are
\begin{align*}
\op{\Inv}((\sg, T)) &= \{(i, j): 1 \leq i < j \leq n, ~ \sg_{i} > \sg_{j}, ~ j-1 \notin S \} .
\end{align*}
This follows from the fact that the non-minimal elements in each block are all to the right of the minimal element when blocks are increasing.
We obtain the statistic
\begin{align}
\label{ascent-haglund}
\op{\inv}((\sg, T)) &= \inv(\sg) - \sum_{i \in T} \inv^{\Box, i+1}(\sg)
\end{align} on $\S^{<}_{n}$. 
These new statistics are, by definition, equidistributed, i.e.\
\begin{equation*}
\dist^{\op{\inv}}_{\S^{>}_{n, n-k}}(q) = 
\dist^{\op{\inv}}_{\S^{<}_{n, n-k}}(q).
\end{equation*}
Furthermore, we can obtain an expression for ascent-starred permutations much like the inversion side of Haglund's conjecture. This yields the identity
\begin{multline*}
\sum_{\sigma \in \S_{n}} q^{\inv(\sigma)} \prod_{i \in \Des(\sigma)} \left(1 + \frac{z}{q^{1 + \inv^{\Box, i}(\sigma)}} \right) = \\
\sum_{\sigma \in \S_{n}} q^{\inv(\sigma)} \prod_{i \in \Asc(\sigma)} \left(1 + \frac{z}{q^{\inv^{\Box, i+1}(\sigma)}} \right) .
\end{multline*}

Next, we consider the right-hand side of Haglund's conjecture (\ref{haglund-conj}), 
\begin{equation}\label{eq:maj-side}
 \sum_{\sigma \in \S_{n}} q^{\maj(\sigma)} \prod_{j = 1}^{\des(\sigma)} \left(1 + \frac{z}{q^{j}} \right).
\end{equation}
In this case, we shall think of the index $j$ in the
product  $\prod_{j = 1}^{\des(\sigma)} \left(1 + \frac{z}{q^{j}} \right)$
as referring  to the descents of $\sg$ as we read from right to left\footnote{Somewhat surprisingly, we \emph{must} consider descents in this order (not left to right) when dealing with the major index.}.
Again, we think of the choice of $1$ from the factor
$1 + \frac{z}{q^{j}}$ as leaving the $j$th descent (from right to left) unstarred and
the choice of $\frac{1}{zq^{j}}$ from the factor
$1 + \frac{z}{q^{j}}$ as starring this descent. We need our statistic to decrease by $j$ when we star this $j$th descent. One way to accomplish this is to have every star subtract the number of descents weakly to its right. With this in mind, we set the \emph{major index} of a descent-starred permutation to be
\begin{align*}
\op{\maj}((\sg, S)) &= \maj(\sg) - \sum_{i \in S} | \Des(\sg) \cap \{i, i+1, \ldots, n-1\} | .
\end{align*}
Alternatively, we could have every descent subtract the number of stars weakly to its left, i.e.\ 
\begin{align*}
\op{\maj}((\sg, S)) &= \sum_{i \in \Des(\sg)} \left( i - |S \cap \{1, 2, \ldots, i\}| \right) \\
&= \maj(\sg) - \sum_{i \in \Des(\sg)} |S \cap \{1, 2, \ldots, i\}| .
\end{align*}
For example, if $(\sg, S) = 7_*3_*2~6~4_*1~5$, the first definition gives
\begin{align*}
\op{\maj}((\sg, S)) &= \maj(\sg) - \sum_{i \in S} | \Des(\sg) \cap \{i, i+1, \ldots, n-1\} | \\
&= 12 - (4 + 3 + 1) = 4.
\end{align*}
To use the second definition, we associate with $(\sg, S)$ a weakly increasing sequence that increments each time we reach an unstarred position. When $(\sg, S) = 7_*3_*2~6~4_*1~5$, this sequence is $(0, 0, 1, 2, 2, 3, 4)$. Then the major index comes from summing the elements of this sequence that correspond to descents in $\sg$, i.e.\
\begin{align*}
\op{\maj}((\sg, S)) &= 0 + 0 + 2 + 2 = 4.
\end{align*}
This point of view makes it clear that, when $S = \emptyset$, our major index reduces to the usual major index for permutations\footnote{It also shows that our major index is similar to the statistic $\operatorname{bmajmil}$ in \cite{stein}. Our main contribution is our bijection between the major index and inversion number.}.

It follows that 
\begin{align*}
\sum_{\sigma \in \S_{n}} q^{\maj(\sigma)} \prod_{j = 1}^{\des(\sigma)} \left(1 + \frac{z}{q^{j}} \right) &= \sum_{k=1}^{n} 
\dist^{\maj}_{\S^{>}_{n, k}}(q) z^{k}.
\end{align*}

Unlike in the inversion case, there seems to be no natural way to define this statistic on ordered set partitions or ascent-starred permutations. In other words, the only way to extend this $\maj$ statistic to either of these sets is to follow the canonical bijections to descent-starred permutations and apply our definition there.

Thus, we have shown that Haglund's conjecture would follow if we could show that $\op{\inv}$ and $\op{\maj}$ were equidistributed over descent-starred permutations, i.e.\
\begin{align*}
\dist^{\op{\inv}}_{\S^{>}_{n, k}}(q) &= \dist^{\op{\maj}}_{\S^{>}_{n, k}}(q).
\end{align*}
Our next task is to generalize the insertion method to give a bijective proof of this statement.

\subsection{Labelings and Insertion Maps}
\label{ssec:gen-lab}

The goal of our generalized insertion lemmas is to prove 
that $\inv$ and $\maj$ on $\S_n^{>}$ satisfy the recursions 
\begin{align*}
\dist^{\op{\inv}}_{\S^{>}_{n, k}}(q) &= [n-k]_{q} \dist^{\op{\inv}}_{\S^{>}_{n-1, k}}(q) + [n-k]_{q} \dist^{\op{\inv}}_{\S^{>}_{n-1, k-1}}(q) \ \mbox{and} \\
\dist^{\op{\maj}}_{\S^{>}_{n, k}}(q) &= [n-k]_{q} \dist^{\op{\maj}}_{\S^{>}_{n-1, k}}(q) + [n-k]_{q} \dist^{\op{\maj}}_{\S^{>}_{n-1, k-1}}(q) .
\end{align*}
As in the $\S_{n}$ case, we will get a recursive bijection between the two statistics as a result.

There are two ways to obtain an element of $\S^{>}_{n, k}$ from some element of $\S^{>}_{n-1}$. The first is to start with an element of $\S^{>}_{n-1, k}$ and to insert $n$ without adding a new star. This is equivalent to saying that the insertion of $n$ adds a new bar (and a new block) to the associated ordered set partition. We will call this type of insertion a \emph{bar insertion}. The second way to create an element of $\S^{>}_{n, k}$ is to start with an element of $\S^{>}_{n-1, k-1}$ and to add a new star while inserting $n$. We will call this type of insertion a \emph{star insertion}. In our rook theory model of 
ordered set partitions, bar insertion corresponds to adding an extra 
column that contains a file rook and star insertion corresponds to 
an extra column that contains a non-attacking rook.

Now we give an $\op{\inv}$-labeling associated with each type of insertion. Take a descent-starred permutation $(\sg, S) \in \S^{>}_{n-1, k}$. We will only   
label positions which follow an element which is not starred 
plus the position at the start of the descent-starred permutation.  
Like the $\op{\inv}$-labeling for $\S_n$, we will label these positions with $0, 1, \ldots, k+1$ from right to left. For a star insertion, we will follow the same procedure, but we will skip the rightmost position. 

For example, say $(\sg, S) = 5~2_*1~4~7_*6_*3$. The $\op{\inv}$-labeling of a bar insertion is
\begin{align}
\label{ex:inv-bar}
 _{4}5_{3}2_*1_{2}4_{1}7_*6_*3_0
\end{align}
and the $\op{\inv}$-labeling of a star insertion is
\begin{align}
\label{ex:inv-star}
 _{3}5_{2}2_*1_{1}4_{0}7_*6_*3 .
\end{align}

Next we define insertion maps for each type of insertion. For bar insertion, we construct the map
\begin{align*}
\phi \insertion{|}{\op{\inv}}{n, k} : \{0, 1, \ldots, n-k-1\} \times \S^{>}_{n-1, k} \to \S^{>}_{n, k}
\end{align*}
by sending $(i, (\sg, S))$ to the ordered set partition where $n$ has been inserted at the position in $(\sg, U)$ that received  the bar insertion $\op{\inv}$-label $i$. For example, with $(sg, S) = 5~2_*1~4~7_*6_*3$, the labeling in (\ref{ex:inv-bar}) implies
\begin{align*}
\phi \insertion{|}{\op{\inv}}{8, 3} (2, (\sg, S)) &= 5~2_*1~8~4~7_*6_*3 .
\end{align*}
Similarly, 
\begin{align*}
\phi \insertion{*}{\op{\inv}}{n, k} : \{0, 1, \ldots, n-k-1\} \times \S^{>}_{n-1, k-1} \to \S^{>}_{n, k}
\end{align*}
sends $(i, (\sg, S))$ to the ordered set partition where $n$ has been inserted and starred at the position labeled $i$ under the $\op{\inv}$-labeling associated with star insertions. We can use the labeling in (\ref{ex:inv-star}) to get
\begin{align*}
\phi \insertion{*}{\op{\inv}}{8, 4}(2, (\sg, S)) &= 5~8_*2_*1~4~7_*6_*3 .
\end{align*}
The next lemma proves that these labels and insertion maps cooperate.

\begin{lemma}
\label{lemma:osp-inv}
\leavevmode
\begin{itemize}
\item For $(\sg, S) \in \S^{>}_{n-1, k}$, $\op{\inv}(\phi \insertion{|}{\op{\inv}}{n, k}(i, (\sg, S))) = \op{\inv}((\sg, S)) + i$.
\item For $(\sg, S) \in \S^{>}_{n-1, k-1}$, $\op{\inv}(\phi \insertion{*}{\op{\inv}}{n, k}(i, (\sg, S))) = \op{\inv}((\sg, S)) + i$.
\end{itemize}
\end{lemma}

\begin{proof}
To prove the first statement, we notice that inserting $n$ at the position that received the label $i$ for bar insertion creates $i$ new inversions (between $n$ and all the unstarred elements to its right) and does not affect any of the previous inversions. The same is true for star insertion, since each star insertion label is one less than the bar insertion label at the same position, and one less inversion is created.
\end{proof}

Furthermore, we claim that these insertion maps give us a unique way of creating every element in $\S^{>}_{n,k}$. That is, $\phi \insertion{|}{\op{\inv}}{n, k}$ and $\phi \insertion{*}{\op{\inv}}{n, k}$ are both injections, and their respective images are
\begin{align*}
\{(\tau, T) \in \S^{>}_{n,k} : n \text{ is not starred in } (\tau, T) \}
\end{align*}
and
\begin{align*}
\{(\tau, T) \in \S^{>}_{n,k} : n \text{ is starred in } (\tau, T) \} .
\end{align*} Clearly the (disjoint) union of these two sets is $\S^{>}_{n,k}$. Therefore $\op{\inv}$ satisfies the following recursion.

\begin{prop}
\label{prop:osp-inv-dist}
The polynomial $\dist^{\op{\inv}}_{\S^{>}_{n,k}}$ equals 1 when $k = n-1$, 0 when $k < 0$ or $k > n-1$, and
\begin{align*}
\dist^{\op{\inv}}_{\S^{>}_{n,k}}(q) &= [n-k]_{q} \dist^{\op{\inv}}_{\S^{>}_{n-1,k}}(q) + [n-k]_{q} \dist^{\op{\inv}}_{\S^{>}_{n-1,k-1}}(q)
\end{align*}
otherwise.
\end{prop}

\begin{proof}
It follows from  Lemma \ref{lemma:osp-inv} plus our discussion of 
the properties of $\phi \insertion{|}{\op{\inv}}{n,k}$ and $\phi \insertion{*}{\op{\inv}}{n,k}$ that 
\begin{align*}
\dist^{\op{\inv}}_{\S^{>}_{n,k}}(q) &= \sum_{(\tau, T) \in \S^{>}_{n,k}} q^{\op{\inv}((\tau, T))} \\
&= \sum_{\stackrel{(\tau, T) \in \S^{>}_{n,k}}{n \ \text{is not starred in } (\tau, T)}} q^{\op{\inv}((\tau, T))} +  \sum_{\stackrel{(\tau, T) \in \S^{>}_{n,k}}{n \text{ is starred in } (\tau, T)}} q^{\op{\inv}((\tau, T))} \\
&= \sum_{i=0}^{n-k-1} \sum_{(\sg, S) \in \S^{>}_{n-1, k}} q^{\op{\inv}(\phi \insertion{|}{n}{k}(i, (\sg, S)))} + \\
& \ \ \ \ \sum_{i=0}^{n-k-1} \sum_{(\sg, S) \in \S^{>}_{n-1, k}} q^{\op{\inv}(\phi \insertion{*}{n}{k}(i, (\sg, S)))} \\
&= \sum_{i=0}^{n-k-1} \sum_{(\sg, S) \in \S^{>}_{n-1, k}} q^{\op{\inv}((\sg, S)) + i} + \sum_{i=0}^{n-k-1} \sum_{(\sg, S) \in \S^{>}_{n-1, k-1}} q^{\op{\inv}((\sg, S)) + i} \\
&= [n-k]_{q} \dist^{\op{\inv}}_{\S^{>}_{n-1, k}}(q) + [n-k]_{q} \dist^{\op{\inv}}_{\S^{>}_{n-1, k}}(q) .
\end{align*}
\end{proof}

We want to show that the major index has the same distribution as the inversion number on $\S^{>}_{n, k}$. In order to prove this we rely on the insertion lemma for the major index over the symmetric group, since the only way we know how to calculate the major index of descent-starred permutation $(\sg, S)$ uses the major index of the permutation $\sg$.

As before, we first define labelings that echo the labelings from the symmetric group case. For a bar insertion, we label the rightmost position in our descent-starred permutation with a zero, and then label its unstarred descents from right to left with $1, 2, \ldots$. We label the leftmost position with the next number, and then label the unstarred ascents from left to right with increasing labels. This gives us the $\op{\maj}$-labeling of a bar insertion. For star insertions, we skip the rightmost position and then follow the same procedure.

For example, say $(\sg, S) = 5~2_*1~4~7_*6_*3$. The $\op{\maj}$-labeling of a bar insertion for $(\sg, S)$ is
\begin{align}
\label{ex:maj-bar}
 _{2}5_{1}2_*1_{3}4_{4}7_*6_*3_{0}
\end{align}
and the $\op{\maj}$-labeling of a star insertion for $(\sg, S)$ is
\begin{align}
\label{ex:maj-star}
 _{1}5_{0}2_*1_{2}4_{3}7_*6_*3 .
\end{align}

Now we need to define how to build a new descent-starred permutation after choosing a certain label. This process will be quite different from the process we established for the inversion number. For bar insertion we will define the map
\begin{align*}
\phi \insertion{|}{\op{\maj}}{n, k} : \{0, 1, \ldots, n-k-1\} \times \S^{>}_{n-1,k} \to \S^{>}_{n,k}
\end{align*}
by sending $(i, (\sg, S))$ to the descent-starred permutation obtained from $(\sg, S)$ by
\begin{enumerate}
\item inserting $n$ at the $\op{\maj}$-label $i$ associated with bar insertion, and then
\item moving each star to the right of $n$ one descent to its left.
\end{enumerate}
This second step is well-defined because, after step 1, $n$ will always be an unstarred descent. It also follows that every star that has just moved will remain weakly to the right of the $n$, and the rightmost descent will be unstarred. To obtain
\begin{align*}
\phi \insertion{*}{\op{\maj}}{n, k} : \{0, 1, \ldots, n-k-1\} \times \S^{>}_{n-1,k-1} \to \S^{>}_{n,k}
\end{align*}
we follow steps 1 and 2 and then place a star at the rightmost descent of the resulting descent-starred permutation.

For example, with $(\sg, S) = 5~2_*1~4~7_*6_*3$ as above, (\ref{ex:maj-bar}) gives
\begin{align*}
\phi \insertion{|}{\op{\maj}}{8, 5} (3, (\sg, S)) &= 5~2_*1~8_*4~7_*6~3
\end{align*}
and (\ref{ex:maj-star}) gives
\begin{align*}
\phi \insertion{*}{\op{\maj}}{8, 4} (3, (\sg, S)) =  5~2_*1~4~8_*7_*6_*3 .
\end{align*}

As in the inversion case, we show that these labels and insertion maps cooperate.

\begin{lemma}
\label{lemma:osp-maj}
\leavevmode
\begin{itemize}
\item For $(\sg, S) \in \S^{>}_{n-1, k}$, $\op{\maj}(\phi \insertion{|}{\op{\maj}}{n, k}(i, (\sg, S))) = \op{\maj}((\sg, S)) + i$.
\item For $(\sg, S) \in \S^{>}_{n-1, k-1}$, $\op{\maj}(\phi \insertion{*}{\op{\maj}}{n, k}(i, (\sg, S))) = \op{\maj}((\sg, S)) + i$.
\end{itemize}
\end{lemma}

\begin{proof}
We will focus on the first statement as the second statement will be a simple consequence of the first. 

Let $(\tau,T) = \phi \insertion{|}{\op{\maj}}{n-1,i}(\sg,S)$. 
Recall that 
\begin{align*}
\op{\maj}((\sg, S)) &= \maj(\sg) - \sum_{i \in S} |S \cap \{1, \ldots,i\} | .
\end{align*}

If $i = 0$, we insert $n$ at the far right end. This does not change either term in the above expression, so we have $\op{\maj}(\phi \insertion{|}{\op{\maj}}{n, k}(0, (\sg, S))) = \op{\maj}((\sg, S))$. 

Now suppose that the space labeled $i$ under the maj-labeling 
of $(\sg,S)$ is the space immediately following $\sg_p$ where 
$\sg_p > \sg_{p+1}$ and $\sg_p \not \in S$. 
Suppose that there are $a$ starred descents and $b$ unstarred 
descents to the left of $\sg_p$ and  
$c$ starred descents and $d$ unstarred 
descents strictly to the right of $\sg_p$ in $(\sg,S)$. 
Then the space following 
$\sg_p$ is labeled with $i=d+1$ in our maj-labeling 
of $(\sg,S)$ and it is labeled with $c+d+1$ in 
maj-labeling of $\sg$. 
Thus $\tau$ is  the permutation 
that arises by inserting $n$ immediately after $\sg_p$ 
and $T$ be the set of starred elements that 
is the result of moving the stars on the descents to the right 
$n$ one descent to the left in $\tau$.  
First we claim that  
$$\sum_{i \in \Des(\sg)} |S\cap \{1,\ldots, i\}|= \sum_{i \in \Des(\tau)} |S\cap \{1,\ldots, i\}|.$$
We claim that  before we move the 
stars on the descents to right of $n$ one descent to left, 
the insertion of $n$ does not change 
$\sum_{i \in \Des(\sg)} |S\cap \{1,\ldots, i\}|$. 
That is, before insertion of $n$, $p \in \Des(\sg)$ and 
$|S \cap \{1, \ldots, p\}| =a$ 
while after the insertion of $n$, $p \notin \Des(\tau)$ 
but $p+1 \in \Des(\tau)$ and there will still be $a$ starred elements 
weakly to the left of position $p+1$. The  insertion 
of $n$ does not effect number of starred descents weakly to 
the left for any other descent in $\sg$. Next observe that the effect of 
moving the star on any given descent, one descent to left increases the 
corresponding sum 
$\sum_{i \in \Des(\tau)} |S \cap \{1, \ldots, i\}|$ by one.  Since we 
are moving $c$ stars, we have that 
$c+\sum_{i \in \Des(\sg)} |S \cap \{1, \ldots, i\}|=
\sum_{i \in \Des(\tau)} |T \cap \{1, \ldots, i\}|$.
Hence, 
\begin{eqnarray*}
\maj((\tau,T)) &=& \maj(\tau) - \sum_{i \in \Des(\tau)} |T \cap \{1, \ldots, i\}| \\
&=& 1+c+d+\maj(\sg) - (c + \sum_{i \in \Des(\sg)} |S \cap \{1, \ldots, i\}|)\\
&=& 1+d +\maj(\sg) - \sum_{i \in \Des(\sg)} |S \cap \{1, \ldots, i\}| \\
&=& 1+d +\maj((\sg,S)) = i + \maj((\sg,S)).
\end{eqnarray*}

Next, suppose that the space labeled $i$ under the maj-labeling of 
$(\sg,S)$ is the space $s$ at the start of $(\sg,S)$.  
Assume that are $c$ starred descents 
and $d$ unstarred descents in $(\sg,S)$.  Then, under 
maj-labeling of $(\sg,S)$, $s$ has label $i=d+1$ and, under the 
maj-labeling of $\sg$, $s$ has label $c+d+1$. 
Let $\tau$ be  the permutation 
that arises by inserting $n$ at the start of $(\sg,S)$  
and $T$ be the set of starred elements that 
is the result of moving the stars on the descents to the right 
$n$ one descent to the left in $\tau$.  Then, 
as above, we can argue that 
$c+\sum_{i \in \Des(\sg)} |S \cap \{1, \ldots, i\}|=
\sum_{i \in \Des(\tau)} |T \cap \{1, \ldots, i\}|$ so that 
 $\maj((\tau,T)) = 1+d  + \maj((\sg,S)=i + \maj((\sg,S)$.

Now suppose that the space labeled $i$ under the maj-labeling 
of $(\sg,S)$ is the space following $\sg_p$ where 
$\sg_p < \sg_{p+1}$ so that $\sg_p$ is not starred 
in $(\sg,S)$. Suppose that there are $a$ starred descents 
and $b$ unstarred descents in $(\sg,S)$ strictly to the 
left of $\sg_p$  and $c$ unstarred descents and $d$ starred 
descents in $(\sg,S)$ strictly to the 
right of $\sg_p$.  Then, under the maj-labeling for $(\sg,S)$, the space 
at the start of the permutation is labelled with $1+b+d$ and 
hence the space after $\sg_p$ is labeled with 
$1+b+d +(p-(a+b)) = 1+d-a +p =i$.  Under the $\maj$-labeling for $\sg$,
the space 
at the start of the permutation is labelled with $1+a+b+c+d$ and 
hence the space after $\sg_p$ is labeled with 
$1+a+b+c+d +(p-(a+b)) = 1+c+d +p$.
Then $\tau$ is the permutation 
that arises by inserting $n$ immediately after $\sg_p$ 
and $T$ be the set of starred elements that 
is the result of moving the stars on the descents to the right 
$n$ one descent to the left in $\tau$. 
Before we move the 
stars on the descents to right of $n$ one descent to left, 
the insertion of $n$ does not change the number of 
starred elements weakly to left of any descent in $\sg$. 
However, $n = \tau_{p+1}$ is now a new descent and 
there are $a$ starred element strictly to the left of $n$. 
As before, moving the stars on the $c$ starred descents to the 
right of $n$ one descent to the left gives an addition 
contribution of $c$ to $\sum_{i \in \Des(\tau)} |T \cap \{1, \ldots, i\}|$. 
It follows that
$$a+c + \sum_{i \in \Des(\sg)} |S \cap \{1, \ldots, i\}| = 
\sum_{i \in \Des(\tau)} |T \cap \{1, \ldots, i\}|.$$
Hence
\begin{eqnarray*}
\maj((\tau,T)) &=& \maj(\tau) -  \sum_{i \in \Des(\tau)} |T \cap \{1, \ldots, i\}|\\
&=& 1+c+d+p+\maj(\sg) - 
(a+c + \sum_{i \in \Des(\sg)} |S \cap \{1, \ldots i\}|)\\
&=& 1+d+p +(\maj(\sg) - \sum_{i \in \Des(\sg)} |S \cap \{1, \ldots i\}|) \\
&=& 1+d+p-a +\maj((\sg,S))= i+\maj((\sg,S)).
\end{eqnarray*}

For the second statement in the lemma, we observe that each position which 
receives the label 
$i+1$ during bar insertion receives label $i$ in star insertion 
of $i \geq 0$. Therefore after steps 1 and 2 of the star insertion 
procedure, inserting $n$ into the position labeled $i$ has 
increased the major index by $i+1$. Starring the rightmost descent 
then decreases the major index by 1 so that we get the desired result. 
\end{proof}

As in the $\op{\inv}$ case, this lemma allows us to prove a recursion for the statistic $\op{\maj}$. The images of $\phi \insertion{|}{\op{\maj}}{n,k}$ and $\phi \insertion{*}{\op{\maj}}{n,k}$  are
\begin{align*}
\{(\tau, T) \in \S^{>}_{n, k} : \text{rightmost descent is not starred in } (\tau, T) \}
\end{align*}
and
\begin{align*}
\{(\tau, T) \in \S^{>}_{n, k} : \text{rightmost descent is starred in } (\tau, T) \},
\end{align*}
respectively. As before, the disjoint union of these two sets is $\S^{>}_{n, k}$. Furthermore, both of these maps are injective. To see that $\phi \insertion{*}{\op{\maj}}{n,k}$ is injective, note that its inverse is equal to
\begin{enumerate}
\item removing the star on the rightmost descent,
\item moving every star weakly to the right of $n$ one descent to its right, and
\item removing $n$ and recording $i$ as the difference in $\op{\maj}$ between the beginning ordered set partition and the final ordered set partition.
\end{enumerate}
To calculate the inverse of $\phi \insertion{|}{\op{\maj}}{n,k}$ we just skip step 1.

\begin{prop}
\label{prop:osp-maj-dist}
The polynomial $\dist^{\op{\maj}}_{\S^{>}_{n, k}}$ equals 1 when $k = n-1$, 0 when $k < 0$ or $k > n-1$, and
\begin{align*}
\dist^{\op{\maj}}_{\S^{>}_{n, k}}(q) &= [n-k]_{q} \dist^{\op{\maj}}_{\S^{>}_{n-1, k}}(q) + [n-k]_{q} \dist^{\op{\maj}}_{\S^{>}_{n-1, k-1}}(q)
\end{align*}
otherwise.
\end{prop}

\begin{proof}
\begin{align*}
\dist^{\op{\maj}}_{\S^{>}_{n, k}}(q) &= \sum_{(\tau, T) \in \S^{>}_{n, k}} q^{\op{\maj}((\tau, T))} \\
&= \sum_{\stackrel{(\tau, T) \in \S^{>}_{n, k}}{\text{rightmost descent is not starred in } (\tau, T)}} q^{\op{\maj}((\tau, T))} + \\ 
&\ \ \ \   \sum_{\stackrel{(\tau, T) \in \S^{>}_{n, k}}{\text{rightmost descent is starred in } (\tau, T)}} q^{\op{\maj}((\tau, T))} \\
&= \sum_{i=0}^{n-k-1} \sum_{(\sg, S) \in \S^{>}_{n-1, k}} q^{\op{\maj}(\phi \insertion{|}{\op{\maj}}{n,k}(i, (\sg, S)))} + \\
&\ \ \ \  \sum_{i=0}^{n-k-1} \sum_{(\sg, S) \in \S^{>}_{n-1, k-1}} q^{\op{\maj}(\phi \insertion{*}{\op{\maj}}{n,k}(i, (\sg, S)))} \\
&= \sum_{i=0}^{n-k-1} \sum_{(\sg, S) \in \S^{>}_{n-1, k}} q^{\op{\maj}((\sg, S)) + i} + \sum_{i=0}^{n-k-1} \sum_{(\sg, S) \in \S^{>}_{n-1, k-1}} q^{\op{\maj}((\sg, S)) + i} \\
&= [n-k]_{q} \dist^{\op{\maj}}_{\S^{>}_{n-1, k}}(q) + [n-k]_{q} \dist^{\op{\maj}}_{\S^{>}_{n-1, k-1}}(q) .
\end{align*}
\end{proof}

We finally have all the ingredients we need to give a bijective proof of Haglund's conjecture. We recursively define our bijection $\psi^{>}_{n, k}$ by setting $\psi^{>}_{1, 0} : \S^{>}_{1,0} \rightarrow \S^{>}_{1,0}$ equal to the identity map and $\psi^{>}_{n, k} : \S^{>}_{n, k} \rightarrow \S^{>}_{n, k}$ on $(\sg, S)$ by
\begin{align*}
\psi^{>}_{n, k}  = \left\{ \begin{array}{ll}
\phi \insertion{|}{\op{\maj}}{n,k} \circ (\id, \psi^{>}_{n-1, k}) \circ (\phi \insertion{|}{\op{\inv}}{n,k})^{-1} & n \text{ is not starred in } (\sg, S) \\
\phi \insertion{*}{\op{\maj}}{n,k} \circ (\id, \psi^{>}_{n-1, k-1}) \circ (\phi \insertion{*}{\op{\inv}}{n,k})^{-1} & n \text{ is starred in } (\sg, S) .
\end{array} \right.
\end{align*}
This bijection takes the inversion number to the major index and preserves the number of blocks in the ordered set partition. Indeed, assume $n$ is not starred in $(\tau, T)$ and set $(i, (\sg, S)) = (\phi \insertion{|}{\op{\inv}}{n,k})^{-1}((\tau, T))$. Then
\begin{align*}
\op{\inv}((\tau, T)) &= \op{\inv}(\phi \insertion{|}{\op{\inv}}{n, k}(i, (\sg, S))) \\
&= i + \op{\inv}((\sg, S)) \\
&= i + \op{\maj}( \psi^{>}_{n-1, k-1}((\sg, S))) \\
&= \op{\maj}( \phi \insertion{|}{\op{\maj}}{n,k} \left( \psi^{>}_{n-1, k-1}((\sg, S)) \right) ) \\
&= \op{\maj}( \psi^{>}_{n, k}((\tau, T))) .
\end{align*}
The argument is essentially the same if $n$ is starred. We can define $\psi^{>}_{n}$ on all of $\S^{>}_{n}$ by applying $\psi^{>}_{n, k}$ to every descent-starred permutation with $k$ stars. 

Similarly, we can inductively define maps 
$\Gamma^{>}_{n,k}:\S^{>}_{n, k} \rightarrow \mathcal{M}_{n,k}$ and \\
$\Delta^{>}_{n,k}:\S^{>}_{n, k} \rightarrow \mathcal{M}_{n,k}$ such 
that for $(\sg,T) \in \S^{>}_{n, k}$,
\begin{eqnarray}
 \inv((\sg,T)) &=& \unc(\Gamma^{>}_{n,k}(\sg,T)) \  \mbox{and} \\
\maj((\sg,T)) &=& \unc(\Delta^{>}_{n,k}(\sg,T)).
\end{eqnarray}
We begin by establishing insertion maps 
\begin{align*}
\phi \insertion{|}{\unc}{n,k} &: \{0, 1, \ldots, n-k-1\} \times \mathcal{M}_{n-1,k} \to \mathcal{M}_{n}{k} \\ 
\phi \insertion{*}{\unc}{n,k} &: \{0, 1, \ldots, n-k-1\} \times \mathcal{M}_{n-1,k-1} \to \mathcal{M}_{n}{k}.
\end{align*}
Fortunately, these maps are quite simple. Say we begin with $i \in \{0, 1, \ldots, n-k-1\}$ and a mixed placement $P$. We append a new column to the right of the board for $P$, creating the staircase board $St_n$. To form $\phi \insertion{|}{\unc}{n,k} ((i, P))$, we place a file rook in the rightmost column so that it has exactly $i$ uncanceled cells below it. To form $\phi \insertion{|}{\unc}{n,k} ((i, P))$, we place a non-attacking rook in the rightmost column so that it has exactly $i$ uncanceled cells below it. It follows from the definition of $\mathcal{M}_{n,k}$ that these maps are well-defined bijections.

Now we can define $\Gamma^{>}_{n,k}$ and $\Delta^{>}_{n,k}$. Both $\Gamma^{>}_{1,0}$ and  $\Delta^{>}_{1,0}$ map 
$(1, \emptyset)$ to the unique placement in $\mathcal{M}_{1,0}$. 
Now suppose that we have defined 
$\Gamma^{>}_{n-1,j}$ and  $\Delta^{>}_{n-1,j}$ 
for all $0 \leq j \leq n-2$.  Then we define $\Gamma^{>}_{n,k}$ on $(\sg,S)$ by
\begin{align*}
\Gamma^{>}_{n,k} &= \left\{ \begin{array}{ll} 
\phi \insertion{|}{\unc}{n,k} \circ (\id, \Gamma^{>}_{n-1,k}) \circ (\phi \insertion{|}{\inv}{n,k})^{-1} & n \text{ is not starred in } (\sg,S) \\
\phi \insertion{*}{\unc}{n,k} \circ (\id, \Gamma^{>}_{n-1,k-1}) \circ (\phi \insertion{*}{\inv}{n,k})^{-1} & n \text{ is starred in } (\sg,S)  \end{array} \right. 
\end{align*} 
With this definition, we can simply set $\Delta^{>}_{n,k} = \Gamma^{>}_{n,k} \circ (\psi^{>}_{n,k})^{-1}$. We present an example of $\Gamma^{>}_{n,k}$ in Figure \ref{fig:bijection}. At this point we have proved the following theorems.

\fig{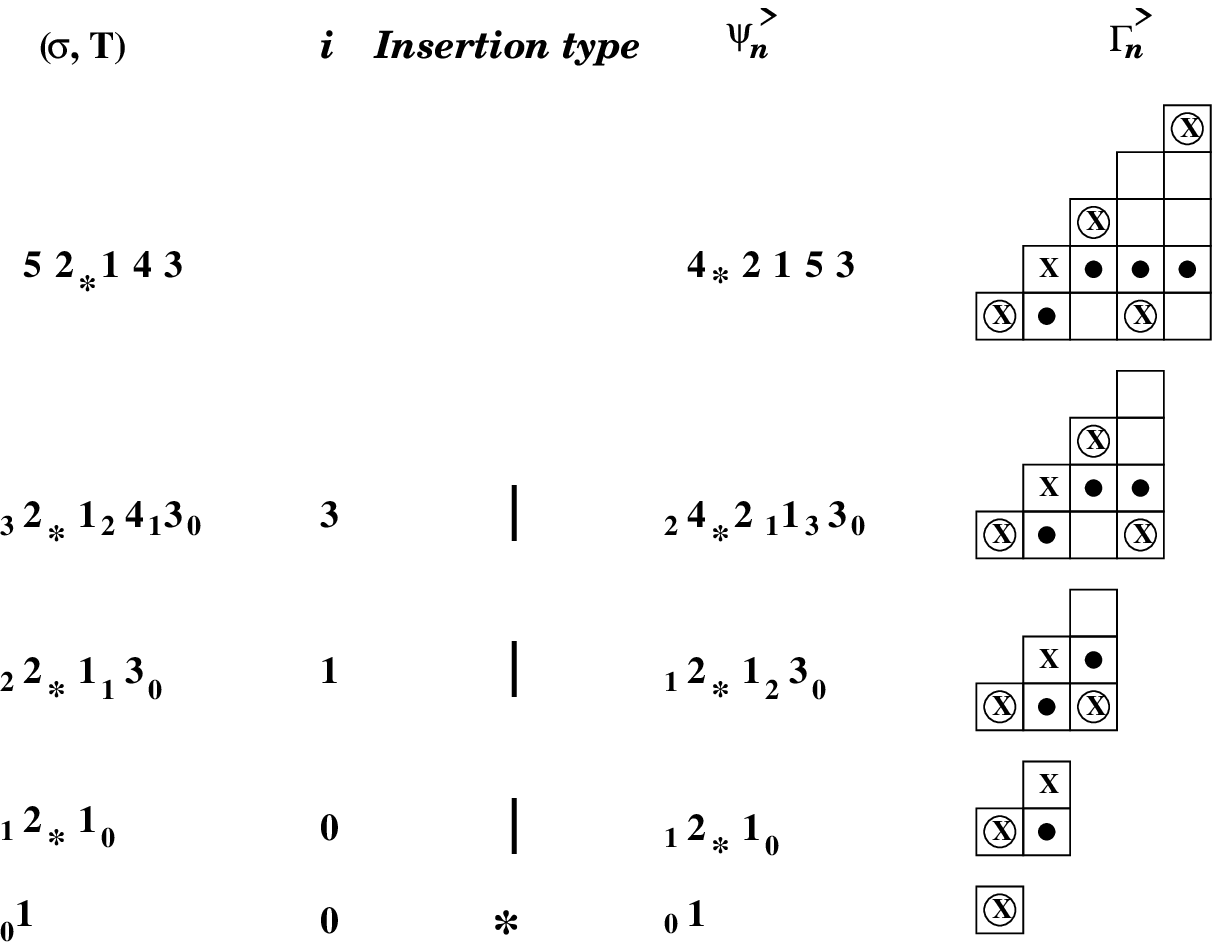}{Examples of the $\psi^{>}_n$ and $\Gamma^{>}_n$ bijections}

\begin{thm} For all $n \geq 1$ and $0 \leq k \leq n-1$, 
\begin{equation*}
\dist^{\inv}_{\S^{>}_{n,k}}(q) = \dist^{\maj}_{\S^{>}_{n,k}}(q) = 
\dist^{\unc}_{\mathcal{M}_{n,k}}(q) = [n-k]_q!S_{n,n-k}(q).
\end{equation*}
\end{thm}

\begin{thm}\label{thm:main} For all $n \geq 1$, 
\begin{align*}
\sum_{k=0}^{n-1} \dist^{\inv}_{\S^{>}_{n,k}}(q) z^{k} &= \sum_{k=0}^{n-1} \dist^{\maj}_{\S^{>}_{n,k}}(q) z^{k} =\sum_{k=1}^n [k]_q! S_{n,k}(q) z^{n-k}\\
&=\sum_{\sigma \in \S_{n}} q^{\inv(\sigma)} \prod_{j \in \Des(\sigma)} \left(1 + \frac{z}{q^{1 + \inv^{\Box, j}(\sigma)}} \right) \\
&=\sum_{\sigma \in \S_{n}} q^{\inv(\sigma)} \prod_{j \in \Asc(\sigma)} \left(1 + \frac{z}{q^{\inv^{\Box, j}(\sigma)}} \right) \\
&= \sum_{\sigma \in \S_{n}} q^{\maj(\sigma)} \prod_{j = 1}^{\des(\sigma)} \left(1 + \frac{z}{q^{j}} \right).
\end{align*}
\end{thm} 

It is worth pointing out that $\psi^{>}$ has a nice property that it shares with Carlitz's bijection on $\S_n$. In any word $\sg$, the \emph{right-to-left minima} of $\sg$ are the entries $\sg_i$ such that, for all $j>i$, $\sg_i < \sg_j$. In other words, they are the entries one marks if one scans $\sg$ from right to left, marking an entry each time it is smaller than all previous entries one has observed. We say that the right-to-left minima of a descent- or ascent-starred permutation are simply the right-to-left minima of the underlying permutation. 

\begin{cor}
\label{cor:right-to-left-min}
For any $(\sg,S) \in \S^{>}_{n}$, $(\sg,S)$ and $\psi^{>}_{n}((\sg,S))$ have the same set of right-to-left minima. As a result, they have the same rightmost entry.
\end{cor}

\begin{proof}
We will prove this fact by induction on $n$. The base case holds trivially. The induction hypothesis is that the statement holds for all values less than $n$. First, we examine the case where $\sg_n = n$. Then the right-to-left minima are the right-to-left minima of $\sg_{1}\ldots\sg_{n-1}$ together with $n$. Let us examine how we create $\psi^{>}_{n}((\sg,S))$. First, we remove $n$ and notice that we have lost zero inversions. Then we apply $\psi^{>}_{n-1}$ to the resulting descent-starred permutation. By induction, this preserves right-to-left minima. Finally, we would like to insert $n$ so that we do not increase the major index or add any new stars. From Lemma \ref{lemma:osp-maj}, we see that the only way to do this is to place $n$ as the new rightmost entry, moving no stars. Therefore we end with the same set of right-to-left minima. 

Now we assume that $\sg_n < n$. In this case, $n$ cannot be a right-to-left minimum of $\sg$, so the right-to-left minima of $\sg$ are exactly the right-to-left minima of $\sg_1 \ldots \sg_{n-1}$. To perform $\psi^{>}_{n}$, we begin by removing $n$ (and the star that might follow it). Then we apply $\psi^{>}_{n-1}$. By induction, $\psi^{>}_{n-1}$ does not alter the set of right-to-left minima. Finally, we insert $n$ according to the maj-labeling. We note that the only way maj-labeling inserts $n$ as the new rightmost entry is if we are performing a bar insertion without increasing the major index. This cannot occur, as it falls under the previous case. Therefore, $n$ is not the rightmost entry in $\psi^{>}_{n}((\sg,S))$, so we have preserved the right-to-left minima.
\end{proof}

\section{Applications and Extensions}
\label{sec:app}
In this section, we explore some consequences of the results obtained in Section \ref{sec:osp}. First, we address a question of Steingr\'{i}msson by proving an alternate form for the distribution of our major index. We extend our methods to other Mahonian statistics, specifically $\coinv$, $\comaj$, $\rlmaj$, and $\rlcomaj$,  in Subsection \ref{ssec:extensions}. We show that 
there are natural $p,q$-analogues of Theorem \ref{thm:main}, in Subsection 
\ref{ssec:pq}. Slightly altering the mixed rook placements of Section 2 leads to new identities in Subsection \ref{ssec:beyond}. Lastly, we mention some future directions of research.

\subsection{Connection to the Euler-Mahonian Distribution}
\label{ssec:stein}
In \cite{stein}, Steingr\'{i}msson explored the distribution $q^{\binom{k}{2}} [k]_{q}! \qstir{n}{k}{q}$ on $\osp{n}{k}$, which he called the \emph{Euler-Mahonian distribution}. He hoped his work would lead to a combinatorial proof of a certain identity. We will show that our major index completes the work begun by Steingr\'{i}msson.

We first need to define the $q$-binomial coefficients
\begin{align*}
\qbinom{n}{k} = \frac{[n]_{q}!}{[k]_{q}! [n-k]_{q}!}
\end{align*}
for $0 \leq k \leq n$. The joint distribution of $\maj$ and $\des$ on $\S_{n}$, i.e.\ $\dist^{\maj, \des}_{\S_{n}}(q, t)$, is known as the \emph{Euler-Mahonian distribution on $\S_{n}$}, and it is standard to write the coefficient of $t^{k}$ in this distribution as $\eul{n}{k}(q)$. It was proved analytically in \cite{zeng-zhang} that
\begin{align}
\label{stein}
q^{\binom{k}{2}} [k]_{q} ! \stir{n}{k}(q) &= \sum_{i=1}^{k} q^{k(k-i)} \qbinom{n-i}{k-i} \eul{n}{i-1}(q) .
\end{align}
Steingr\'{i}msson hoped to come up with a statistic on $\osp{n}{k}$ along with combinatorial proofs that this statistic's distribution was given by each side of the identity. Although he, along with the authors of \cite{osp1, osp2}, proved that many statistics have distributions given by one of the two sides of the equality, they were not able to find a statistic that exhibited both sides. 

Statement (\ref{stein}) is clearly equivalent to
\begin{align}
\label{stein-mah}
[k]_{q} ! \qstir{n}{k}{q} &=  \sum_{i=1}^{k} q^{k(k-i) - \binom{k}{2}} \qbinom{n-i}{k-i} \eul{n}{i-1}(q) .
\end{align}
We proved in Section 4 that
\begin{align*}
\dist^{\op{\maj}}_{\S^{>}_{n, n-k}}(q) &= [k]_{q}! \qstir{n}{k}{q} .
\end{align*}
In this section we prove that the distribution of our major index can also be given by the right-hand side of (\ref{stein-mah}), solving the problem posed by Steingr\'{i}msson.

First, we rewrite
\begin{align*}
A_{n, i-1}(q) &= q^{in - \binom{n+1}{2}} A_{n, n-i}(q) .
\end{align*}
This equality comes from examining how reversing a permutation affects its major index; see \cite{stein} for a full proof. Next, we use the fact that
\begin{align*}
k(k-i) - \binom{k}{2} + in - \binom{n+1}{2} = \binom{n-k}{2} - (n-k)(n-i)
\end{align*}
which can be verified by a straightforward computation. Finally, we use the symmetry of $q$-binomial coefficients to write
\begin{align*}
\qbinom{n-i}{k-i} = \qbinom{n-i}{n-k}
\end{align*}
As a result, (\ref{stein}) is equivalent to the following proposition.

\begin{prop}
\label{prop:stein}
\begin{align*}
\dist^{\op{\maj}}_{\S^{>}_{n, n-k}}(q) &= \sum_{i=1}^{k} q^{\binom{n-k}{2} - (n-k)(n-i)} \qbinom{n-i}{n-k} \eul{n}{n-i}(q).
\end{align*}
\end{prop}

\begin{proof}
We will build a general element $(\sg, S) \in \S^{>}_{n, n-k}$ in a way that exhibits the identity. First we pick the number of ascents of $\sg$. Clearly, 
since we have $n-k$ starred descents, the number of ascents 
must be an element of $\{0, \ldots, k-1\}$. Say that 
$\sg$ has $i-1$ ascents, where $i$ must be in $\{1, \ldots, k\}$. 
Then $\sg$ can be any permutation in $\S_{n}$ with $n-i$ descents. The polynomial $\eul{n}{n-i}(q)$ considers all these possibilities while $q$-counting the major index of the permutation.

Next we must place stars at $n - k$ of the $n- i$ descents in $\sg$. By the definition of $\op{\maj}((\sg, S))$ and the $q$-binomial 
theorem, this choice yields the factor
\begin{align*}
\prod_{j=1}^{n-i} \left. \left(1 + \frac{x}{q^{j}} \right) \right|_{x^{n-k}} = q^{\binom{n-k}{2} - (n-k)(n-i)} \qbinom{n-i}{n-k}.
\end{align*}
\end{proof}

\subsection{Extending Other Mahonian Statistics}
\label{ssec:extensions}
Inversion number and major index are just two of many Mahonian statistics on $\S_{n}$, and it is natural to wonder which of these statistics we can extend to $\osp{n}{k}$ using the methods we have developed. In this section, we will 
 apply a class of well-known bijections to the underlying permutations to obtain generalizations of the statistics $\coinv$, $\comaj$, $\rlmaj$, and $\rlcomaj$. We show that all of these new statistics are equidistributed with $\op{\inv}$ and $\op{\maj}$ on $\S^{>}_{n,k}$.

On $\sigma \in \S_{n}$, we consider the statistics
\begin{align*}
\coinv(\sigma) &= \{1 \leq i < j \leq n: \sigma_{i} < \sigma_{j} \} \\
\comaj(\sigma) &= \sum_{i \in \Asc(\sigma)} i \\
\rlmaj(\sigma) &= \sum_{i \in \Des(\sigma)} n-i \\
\rlcomaj(\sigma) &= \sum_{i \in \Asc(\sigma)} n-i .
\end{align*}
These are known as the \emph{number of coinversions}, \emph{comajor index}, \emph{right-left major index}, and \emph{right-left comajor index}, respectively. Each of these is the image of the inversion number or the major index under one of three simple bijections on $\S_{n}$, sometimes called the trivial bijections: reverse (which sends $\sigma_{i}$ to $\sigma_{n+1-i}$), complement (which replaces $i$ with $n+1-i$), and reverse complement (which is the composition of reverse and complement.) We describe the precise actions of these bijections on the inversion number and major index in the following table. For example, the $\rlcomaj$ entry means that the major index of a permutation is equal to the $\rlcomaj$ of the reverse of that permutation.

\begin{center}
\begin{tabular}{c||c|c|c}
& reverse & complement & reverse complement \\ \hline\hline
$\inv$ & $\coinv$ & $\coinv$ & $\inv$ \\ \hline
$\maj$ & $\rlcomaj$ & $\comaj$ & $\rlmaj$
\end{tabular}
\end{center}

The trivial bijections tell us how to extend these statistics to ordered set partitions. Namely, for an ascent-starred or descent-starred permutation $(\sg, S)$, we apply the bijections to the underlying permutation $\sg$ and reflect $S$ if necessary. For coinversion number, this results in the statistic
\begin{align*}
\op{\coinv}((\sg, S^{\Des})) &= \coinv(\sg) -  \sum_{i \in S^{\Des}} \inv^{i, \Box}(\sg) \\
\op{\coinv}((\sg, S^{\Asc}))&= \coinv(\sg) - \sum_{i \in S^{\Asc}} 1 + \inv^{i+1, \Box}(\sg)
\end{align*}
for $(\sg, S^{\Des}) \in \S^{>}_{n}$, $(\sg, S^{\Asc}) \in \S^{<}_{n}$. We have used superscripts to help keep track of whether each element is an ascent-starred or descent-starred permutation. Combinatorially, this counts the number of $i < j$ such that $i$'s block is to the left of $j$'s block and $i$ is the minimal element in its block\footnote{This is the $\operatorname{los}$ statistic on ordered set partitions in \cite{stein}.}.

To extend the comajor index, right-left major index, and right-left comajor index to starred permutations, we apply the complement, reverse complement, and reverse bijections, respectively, to a descent-starred permutation. This results in
\begin{align*}
\op{\comaj}((\sg, S^{\Asc})) &= \comaj(\sg) - \sum_{i \in S^{\Asc}} |\Asc(\sg) \cap \{i, i+1, \ldots, n-1\}| \\
\op{\rlmaj}((\sg, S^{\Des})) &= \rlmaj(\sg) - \sum_{i \in S^{\Des}} |\Des(\sg) \cap \{1, 2, \ldots, i\}| \\
\op{\rlcomaj}((\sg, S^{\Asc})) &= \rlcomaj(\sg) - \sum_{i \in S^{\Asc}} |\Asc(\sg)) \cap \{1, 2, \ldots, i\}| .
\end{align*}
It follows from the trivial bijections that each of these statistics is equidistributed with $\op{\inv}$ and $\op{\maj}$ on the relevant starred permutations. It is an interesting open problem to investigate how other Mahonian statistics, such as Denert's statistic, may be extended to ordered set partitions.

\subsection{$p,q$-analogues}
\label{ssec:pq}
Define the standard $p,q$-analogues of $k$ and $k!$ by 
\begin{align*}
[k]_{p, q} &= p^{k-1} + p^{k-2}q + \ldots + pq^{k-2} + q^{k-1} \ \mbox{and} \\
[k]_{p, q}! &=[1]_{p, q} [2]_{p, q} \cdots [k]_{p, q}.
\end{align*}

Wachs and White \cite{ww} defined a $p,q$-analogue of the Stirling 
numbers of the second kind $S_{n,k}(p,q)$ by defining the weight 
$w_{p,q}(P)$ of a rook placement $P \in \mathcal{N}_{n-k}(B_n)$ 
as follows. First, each rook cancels all the cells in its row 
to its right plus the cell that it is in. Then we let 
$w_{p,q}(P) = q^{\uncb(P)}p^{\unca(P)}$ where $\uncb(P)$ is the number 
of uncanceled cells that lie below a rook in $P$ and $\unca(P)$ is the number 
of uncanceled cells that lie above a rook in $P$. For example, 
if $P$ is the rook placement pictured on the top left 
in Figure \ref{fig:pq}, we have placed $q$s in cells that contribute 
to $\uncb(P)$ and $p$s in cells that contribute to $\unca(P)$ so 
that $w_{p,q}(P) = p^2q^3$.  Then 
we define $S_{n,k}(p,q) = \sum_{P \in \mathcal{N}_{n-k}(B_n)} w_{p,q}(P)$. 
One can show that the $S_{n,k}(p,q)$ can also be defined by 
the recursions 
\begin{equation}\label{recSnkpq}
S_{n+1,k}(p,q)  = S_{n,k-1}(p,q) +[k]_{p,q}S_{n,k}(p,q).
\end{equation}
with initial conditions that  $S_{0,0}(p,q) =1$ and $S_{n,k}(p,q) =0$ 
if either $k < 0$ or $k > n$.  For example, one can prove 
(\ref{recSnkpq}) by classifying the non-attacking rook placements 
in $\mathcal{N}_{n+1-k}(B_{n+1})$ by according to whether there is 
a rook in the last column.  That is, if $P$ does not have a rook 
in the last column, then $P$ must have $n+1-k$ rooks in $B_n$ so 
that the contribution of such rook placements to $S_{n+1,k}(p,q)$ is 
$$\sum_{P \in \mathcal{N}_{n-(k-1)}(B_n)} w_{p,q}(P) = S_{n,k-1}(p,q).$$
If $P$ does have a rook in the last column, then it has 
$n -k$ rooks in $B_n$ which cancel $n-k$ cells in the last column. 
Thus there are $k$ uncanceled cells in the last column and 
if we place the rook in the last column 
in the $i$th uncanceled cell, reading from top to bottom, 
then we will get a contribution of $q^{i-1}p^{k-1-i}$ to the 
weight of $P$. As a result, the contribution of such rook placements 
to $S_{n+1,k}(p,q)$ is 
$$(p^{k-1} +p^{k-2}q + \cdots + pq^{k-2} +q^{k-1})
\sum_{P \in \mathcal{N}_{n-k}(B_n)} w_{p,q}(B) = [k]_{p,q}S_{n,k}(p,q).$$

\fig{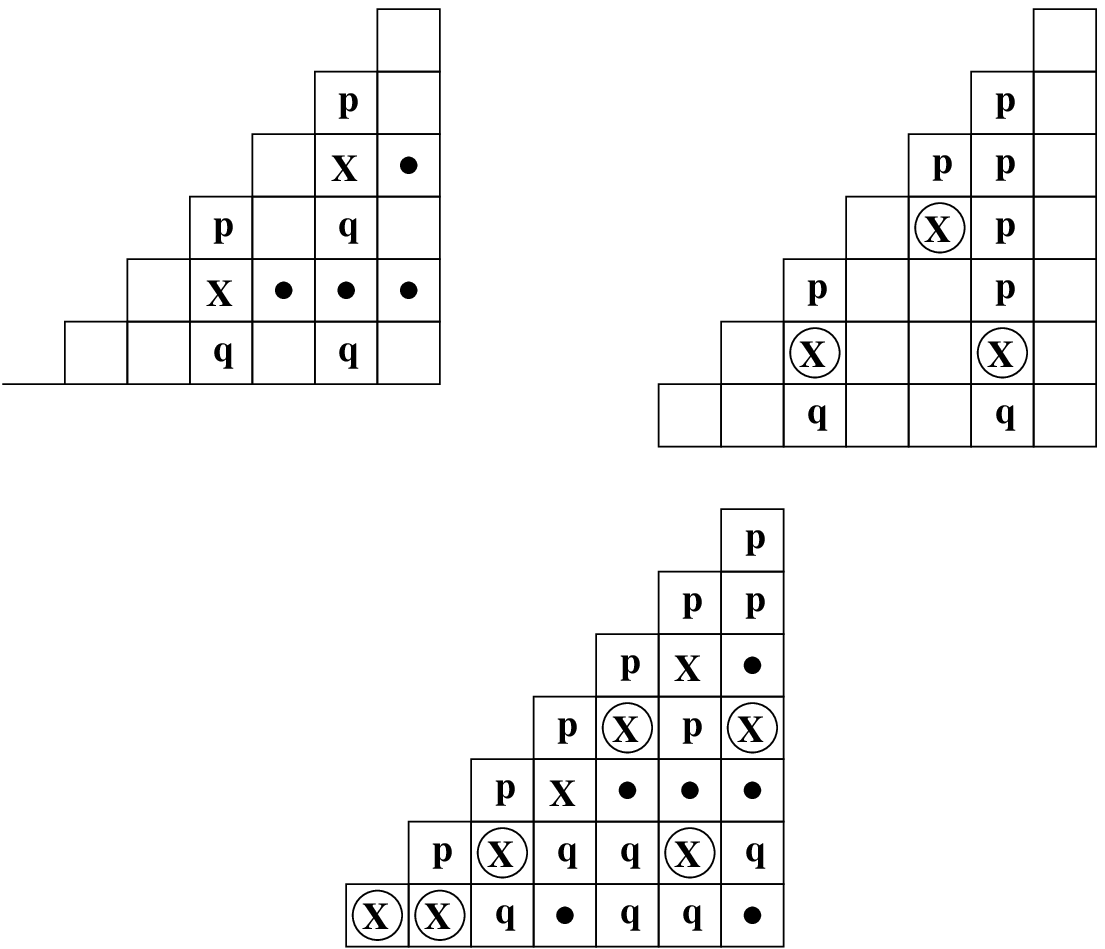}{$p,q$-weights of rook placements, file placements, and 
mixed placements}

We obtain a second $p,q$-analogue of the $S_{n,k}$ by setting  
$$\tilde{S}_{n,k}(p,q) = p^{\binom{n}{2}-\binom{k}{2}}S_{n,k}(q/p).$$ 
We can also give a combinatorial interpretation to  $\tilde{S}_{n,k}(p,q)$.
That is, we know that if we place $n-k$ non-attacking rooks in 
$B_n$, then in the empty columns, we will have shown that 
the number of uncanceled cells is $1, \ldots, k-1$ as we read 
from left to right.  Thus we can view $p^{\binom{n}{2} -\binom{k}{2}}$ 
as placing a weight of $p$ in every cell except those uncanceled 
cells in the empty columns. Then to account for 
the factor $S_{n,k}(q/p)$, we must multiply the weight 
of any uncanceled cell that lies below a rook by $q/p$ which 
will give it an effective weight of $q$. Thus the weight of 
such a placement will be $W_{p,q}(P) = q^{\uncb(P)}p^{\unca(P)}p^{\can(P)}$ 
where $\can(P)$ is the number of canceled cells of $P$. For example, 
For example, if $P$ is the rook placement pictured on the top left 
in Figure \ref{fig:pq}, there are 6 canceled cells so that 
$W_{p,q}(P) = q^3 p^8$. Then 
\begin{equation}
\tilde{S}_{n,k}(p,q) = p^{\binom{n}{2}-\binom{k}{2}}S_{n,k}(q/p) = 
\sum_{P \in \mathcal{N}_{n-k}(B_n)}W_{p,q}(P).
\end{equation}
The same argument that we used to show that 
the recursion (\ref{recSnkpq}) holds will show that 
\begin{equation}\label{rectildeSnkpq}
\tilde{S}_{n+1,k}(p,q) = p^{n+1-k}\tilde{S}_{n,k-1}(p,q)+p^{n+1-k}[k]_{p,q}\tilde{S}_{n,k}(p,q)
\end{equation}
because there are always $n+1-k$ canceled cells in the last 
column for any $P \in \mathcal{N}_{n+1-k}(B_{n+1})$.

Given a file placement $F \in \mathcal{F}(St_n)$, we also 
define $w_{p,q}(F) = q^{\uncb(P)}p^{\unca(P)}$. That is 
the only cells that get canceled in a file placement are 
the cells that contain rooks so that in 
$w_{p,q}(F)$ we are counting a factor of $p$ for every cell that lies 
above a rook and a factor of $q$ for every cell that lies below a rook. 
For example, if $F$ is the file placement pictured on the top right
in Figure \ref{fig:pq}, $w_{p,q}(F) = q^2p^6$.  
It is then easy to see that a rook in a row $j$ and column 
$i$ of $St_n$ contributes $q^{j-1}p^{i-j}$ so that the set of possible 
placements of a rook 
in column $i$ contributes $p^{i-1}+ q p^{i-2}+ \cdots + pq^{i-2}+q^{i-1} = 
[i]_{p,q}$ to $\sum_{F \in \mathcal{F}_n(St_n)} w_{p,q}(F)$. Thus 
\begin{equation}
[n]_{p,q}! = \sum_{F \in \mathcal{F}_n(St_n)} w_{p,q}(F).
\end{equation}

We can similarly define the weight $w_{p,q}(P)$ of a mixed placement 
$P \in \mathcal{M}_{n,k}$ by setting 
$w_{p,q}(P) = q^{\unca(P)} p^{\unca(P)}$, where each 
non-attacking rook or file rook cancels its cell and each 
non-attacking rook cancels all the cells in its row to its right plus 
the cell in its column in the first row. For example, 
if $P$ is the mixed placement pictured in the bottom row 
of Figure \ref{fig:pq}, then we have placed a $q$ in each cell 
counted by $\uncb(P)$ and $p$ in each cell counted by $\unca(P)$ 
so that $w_{p,q}(P) = q^6p^7$. It follows 
from our arguments above that 
\begin{equation}
 D_{\mathcal{M}_{n,k}}^{w_{p,q}} = \sum_{P \in \mathcal{M}_{n,k}}w_{p,q}(P)
= [n-k]_{p,q}!S_{n,n-k}(p,q).
\end{equation}
In addition, by classifying the mixed placements $P \in 
\mathcal{M}_{n,k}$ according to whether the rook in last 
column is a file rook or a non-attacking rook one can show 
that 
\begin{equation}
D_{\mathcal{M}_{n,k}}^{w_{p,q}} =
[n-k]_{p,q} D_{\mathcal{M}_{n-1,k-1}}^{w_{p,q}} + 
[n-k]_{p,q} D_{\mathcal{M}_{n-1,k}}^{w_{p,q}}.
\end{equation}

We now are in position to give two different 
$p,q$-analogues of Theorem \ref{thm:main}. That is,  
in Theorem \ref{thm:main}, we can replace $q$ by $q/p$ and 
then multiply by $p^{\binom{n}{2}}$. 
Then if we observe that 
$p^{-\binom{k}{2}} [k]_{p,q}! = [k]_{q/p}!$ 
and $p^{\coinv(\sg)}q^{\inv(\sg)} = p^{\binom{n}{2}} 
(q/p)^{\inv(\sg)}$ for any $\sg \in \S_n$, we see that 
\begin{align*}
&\sum_{k=1}^n [k]_{p,q}!\tilde{S}_{n,k}(p,q) z^{n-k} \\
&=\sum_{\sigma \in \S_{n}} p^{\coinv(\sg)} q^{\inv(\sigma)} 
\prod_{j \in \Des(\sigma)} \left(1 + \frac{z}{(q/p)^{1 + \inv^{\Box, j}(\sigma)}} \right) \\
&=\sum_{\sigma \in \S_{n}} p^{\coinv(\sg)} q^{\inv(\sigma)} \prod_{j \in \Asc(\sigma)} \left(1 + \frac{z}{(q/p)^{\inv^{\Box, j}(\sigma)}} \right) \\
&= \sum_{\sigma \in \S_{n}} p^{\binom{n}{2}- \maj(\sigma)} q^{\maj(\sigma)} \prod_{j = 1}^{\des(\sigma)} \left(1 + \frac{z}{(q/p)^{j}} \right)
\end{align*}

Alternatively, a more subtle $p,q$-analogue of Theorem 
\ref{thm:main} is the following:
\begin{align*}
&\sum_{\sigma \in \S_{n}} p^{\coinv(\sigma)} q^{\inv(\sigma)} \prod_{j \in \Des(\sigma)} \left(1 + \frac{z}{p^{\coinv^{j, \Box}(\sigma)}q^{1 + \inv^{\Box, j}(\sigma)}} \right) \\
= &\sum_{\sigma \in \S_{n}}p^{\coinv(\sigma)} q^{\inv(\sigma)} \prod_{j \in \Asc(\sigma)} \left(1 + \frac{z}{p^{1 + \coinv^{j, \Box}(\sigma)}q^{\inv^{\Box, j}(\sigma)}} \right) \\
= &\sum_{k=0}^{n-1} \dist^{\op{\coinv}, \op{\inv}}_{\S^{>}_{n, k}}(p, q) z^{k} \\
= &\sum_{k = 1}^{n} [k]_{p, q}! \qstir{n}{k}{p, q} z^{n-k}. 
\end{align*}

The proofs of these identities are very similar to the proof 
of Haglund's conjecture. That is, it follows from our remarks above 
that  
\begin{align*}
&\sum_{\sigma \in \S_{n}} p^{\coinv(\sigma)} q^{\inv(\sigma)} \prod_{j \in \Des(\sigma)} \left(1 + \frac{z}{p^{\coinv^{j, \Box}(\sigma)}q^{1 + \inv^{\Box, j}(\sigma)}} \right) \\
= &\sum_{\sigma \in \S_{n}}p^{\coinv(\sigma)} q^{\inv(\sigma)} \prod_{j \in \Asc(\sigma)} \left(1 + \frac{z}{p^{1 + \coinv^{j, \Box}(\sigma)}q^{\inv^{\Box, j}(\sigma)}} \right) \\
= &\sum_{k=0}^{n-1} \dist^{\op{\coinv}, \op{\inv}}_{\S^{>}_{n, k}}(p, q) z^{k}. \end{align*}
Thus we must show that 
\begin{equation*}
\sum_{k=0}^{n-1} \dist^{\op{\coinv}, \op{\inv}}_{\S^{>}_{n, k}}(p, q) z^{k} \\
=\sum_{k = 1}^{n} [k]_{p, q}! \qstir{n}{k}{p, q} z^{n-k} .
\end{equation*}
This requires only that we show that 
\begin{equation}\label{eq:keypq}
\dist^{\op{\coinv}, \op{\inv}}_{\S^{>}_{n, k}}(p, q) =
[n-k]_{p,q}\dist^{\op{\coinv}, \op{\inv}}_{\S^{>}_{n-1, k-1}}(p, q) + 
[n-k]_{p,q}\dist^{\op{\coinv}, \op{\inv}}_{\S^{>}_{n-1, k}}(p, q).
\end{equation}
To prove (\ref{eq:keypq}), we need only show that 
our $\inv$-bar and $\inv$-star insertions cooperate 
with our coinversion statistic. That is, we must prove the following 
lemma. 
\begin{lemma}
\label{lemma:osp-invpq}
\leavevmode
\begin{itemize}
\item For $(\sg, S) \in \S^{>}_{n-1, k}$, $\op{\coinv}(\phi \insertion{|}{\op{\inv}}{n, k}(i, (\sg, S))) = \op{\coinv}((\sg, S)) + n-k-1-i$.
\item For $(\sg, S) \in \S^{>}_{n-1, k-1}$, $\op{\coinv}(\phi \insertion{*}{\op{\inv}}{n, k}(i, (\sg, S))) = \op{\coinv}((\sg, S)) + n-k-1-i$.
\end{itemize}
\end{lemma}
\begin{proof}
To prove the first statement, we notice that inserting $n$ at the position that received the label $i$ for bar insertion creates $n-k-1-i$ new coinversions (between $n$ and all the unstarred elements to its left) and does not affect any of the previous coinversions. The same is true for star insertion, since each star insertion label is one less than the bar insertion label at the same position, and one less coinversion is created.
\end{proof}

\begin{prop}
\label{prop:osp-inv-distpq}
The polynomial $\dist^{\op{\coinv},\op{\inv}}_{\S^{>}_{n,k}}(p,q)$ equals 1 when $k = n-1$, 0 when $k < 0$ or $k > n-1$, and
\begin{align*}
\dist^{\op{\coinv},\op{\inv}}_{\S^{>}_{n,k}}(p,q) &= [n-k]_{p,q} \dist^{\op{\coinv}, \op{\inv}}_{\S^{>}_{n-1,k}}(p,q) + [n-k]_{p,q} \dist^{\op{\coinv}, \op{\inv}}_{\S^{>}_{n-1,k-1}}(p,q)
\end{align*}
otherwise. 
\end{prop}

\begin{proof}
Our proof is very similar to the proof of 
Proposition \ref{prop:osp-inv-dist}.
By Lemmas \ref{lemma:osp-inv} and \ref{lemma:osp-invpq}, we 
have that 
\begin{align*}
\dist^{\op{\coinv, \inv}}_{\S^{>}_{n,k}}(q) &= \sum_{(\tau, T) \in \S^{>}_{n,k}} p^{\op{\coinv}((\tau, T))}q^{\op{\inv}((\tau, T))} \\
&= \sum_{\stackrel{(\tau, T) \in \S^{>}_{n,k}}{n \ \text{is not starred in } (\tau, T)}} p^{\op{\coinv}((\tau, T))} q^{\op{\inv}((\tau, T))} + \\
& \ \ \ \  \sum_{\stackrel{(\tau, T) \in \S^{>}_{n,k}}{n \text{ is starred in } (\tau, T)}} 
p^{\op{\coinv}((\tau, T))} q^{\op{\inv}((\tau, T))} \\
&= \sum_{i=0}^{n-k-1} \sum_{(\sg, S) \in \S^{>}_{n-1, k}} 
p^{\op{\coinv}(\phi \insertion{|}{n}{k}(i, (\sg, S)))} q^{\op{\inv}(\phi \insertion{|}{n}{k}(i, (\sg, S)))} + \\
& \ \ \ \ \sum_{i=0}^{n-k-1} \sum_{(\sg, S) \in \S^{>}_{n-1, k}} 
p^{\op{\coinv}(\phi \insertion{*}{n}{k}(i, (\sg, S)))} q^{\op{\inv}(\phi \insertion{*}{n}{k}(i, (\sg, S)))} \\
&= \sum_{i=0}^{n-k-1} \sum_{(\sg, S) \in \S^{>}_{n-1, k}} 
p^{\op{\coinv}((\sg, S)) + n-k-1-i}
q^{\op{\inv}((\sg, S)) + i} + \\
& \ \ \ \ \sum_{i=0}^{n-k-1} \sum_{(\sg, S) \in \S^{>}_{n-1, k-1}} p^{\op{\coinv}((\sg, S)) + n-k-1-i} q^{\op{\inv}((\sg, S)) + i} \\
&= [n-k]_{p,q} \dist^{\op{\coinv,\inv}}_{\S^{>}_{n-1, k}}(p,q) + [n-k]_{p,q} \dist^{\op{\coinv,\inv}}_{\S^{>}_{n-1, k}}(p,q) .
\end{align*}
\end{proof}

One may notice that the major index is absent from our second
$p,q$-analogue.  
We can define a companion $\overline{\maj}$ for $\maj$ in 
$\S^{>}_{n, k}$ which would give us a pair of statistics 
on $\S^{>}_{n, k}$ such that 
\begin{equation*}
\sum_{k=0}^{n-1} \dist^{\op{\overline{\maj}}, \op{\maj}}_{\S^{>}_{n, k}}(p, q) z^{k} =\sum_{k = 1}^{n} [k]_{p, q}! \qstir{n}{k}{p, q} z^{n-k}.
\end{equation*}
All we have to do is to ensure that we define 
$\overline{maj}$ so that the following lemma holds. 
\begin{lemma}
\label{lemma:osp-majpq}
\leavevmode
\begin{itemize}
\item For $(\sg, S) \in \S^{>}_{n-1, k}$, $\op{\overline{\maj}}(\phi \insertion{|}{\op{\maj}}{n, k}(i, (\sg, S))) = \op{\overline{\maj}}((\sg, S)) + n-k-i -i$.
\item For $(\sg, S) \in \S^{>}_{n-1, k-1}$, $\op{\overline{\maj}}(\phi \insertion{*}{\op{\maj}}{n, k}(i, (\sg, S))) = 
\op{\overline{\maj}}((\sg, S)) + n-k-1-i$.
\end{itemize}
\end{lemma}

This can be done by simply defining $\overline{\maj}$ by recursion 
so that our maj-bar and maj-star insertion have this property.  The 
problem here is to find a natural definition of this statistic 
which does not refer to the our labelings of spaces. We have 
not been able to find such a definition due to the complications  
that arise by the moving stars in $\maj$-bar and $\maj$-star insertion.

\subsection{Varying Rook Placements}
\label{ssec:beyond}
We mentioned earlier that the conditions about the bottom row in the set of rook placements $\rook_{n, k}$ are unnatural from a rook theoretic point of view. In this section, we see that we can remove these conditions and obtain a
variant of our main theorem that involves a different set of descent-starred permutations. Along the way, we see some how the rook theoretic point of view can help us obtain variations of our main result.

Define $\rook^{\prime}_{n, k}$ to be $\rook_{n, k}$ along with the additional placements that have some right-canceling rook in the bottom row. The natural adaptation of $\unc$ to this set of objects adds in the bottom-row squares below right-canceling rooks. We call this statistic $\unc^{\prime}$. From the rook placements, we observe
\begin{align}
\label{new-recursion}
\dist^{\unc^{\prime}}_{\rook^{\prime}_{n, k}}(q) &=  [n-k+1]_{q} \dist^{\unc^{\prime}}_{\rook^{\prime}_{n-1, k-1}}(q) + [n-k]_{q}  \dist^{\unc^{\prime}}_{\rook^{\prime}_{n-1, k}}(q) \\
&= [n-k]_{q}! \qstir{n+1}{n-k+1}{q} \nonumber.
\end{align}
We would like to adjust our descent-starred permutations and our statistics to obtain this recursion in that setting. These equalities imply that we must replace $\S^{>}_{n, k}$ with some larger set of objects. This set of objects we will consider is
\begin{align*}
\S^{\Falls}_{n, k} &= \{(\sg, S): \sg \in \S_{n}, \ S \subseteq \Des(\sg) \cup \{n\}, \ |S| = k\}.
\end{align*}
By appending a zero to the end of $\sg$ for each $(\sg, S) \in \S^{\Falls}_{n, k}$, we see that we can also think of $\S^{\Falls}_{n, k}$ as the set of permutations of $\{0, 1,\ldots, n\}$ that end in a zero and have $k$ descents starred.

We would like to adjust our statistics so that they match the recursion (\ref{new-recursion}) on $\S^{\Falls}_{n, k}$. In fact, essentially the same statistics work here as before. That is, for $(\sg, S) \in \S^{\Falls}_{n}$, we set
\begin{align*}
\Inv^{\prime}((\sg, S)) &= \{(i, j): 1 \leq i < j \leq n, ~ \sg_{i} > \sg_{j}, ~ j \notin S, ~ \{i, i+1, \ldots, j-1\} \not \subseteq S \} \\
\inv^{\prime}((\sg, S)) &= |\Inv^{\prime}((\sg, S))| = \inv(\sg) - \sum_{i \in S} \inv^{\Box, i}(\sg) .
\end{align*}
The slight difference from the $\S^{>}_{n, k}$ case comes from the fact that it is possible to insert a starred $n$ without creating any inversions, namely by inserting it as far right as possible. Similarly, we set
\begin{align*}
\maj^{\prime}((\sg, S)) &= \maj(\sg) - \sum_{i \in S} |\Des(\sg) \cap \{i, i+1, \ldots, n-1\}|
\end{align*}
for $(\sg, S) \in \S^{\Falls}_{n}$. We claim that the natural adjustments of our insertion procedure to these new objects yields the correct recursions, so we have
\begin{align*}
\dist^{\inv^{\prime}}_{\S^{\Falls}_{n, k}}(q) &= \dist^{\maj^{\prime}}_{\S^{\Falls}_{n, k}}(q) = \dist^{\unc^{\prime}}_{\rook^{\prime}_{n, k}}(q) .
\end{align*}
We can rewrite these identities as 
\begin{align*}
&\sum_{\sg \in \S_{n}} q^{\inv(\sg)} \prod_{i \in \Des(\sg) \cup \{n\}} \left( 1 + \frac{z}{q^{\inv^{\Box, i}(\sg)}} \right) \\
= &\sum_{\sg \in \S_{n}} q^{\maj(\sg)} \prod_{j=0}^{\des(\sg)} \left(1 + \frac{z}{q^{i}} \right) \\
= &\sum_{k=0}^{n} [n-k]_{q}! \qstir{n+1}{n-k+1}{q} z^{k} .
\end{align*}
The second and third lines are each clearly equal to $1+z$ times their corresponding terms in the $\S^{>}_{n, k}$ case. We can also make this $1+z$ term evident in the first line by considering the following. Given an element $(\sg, S) \in \S^{\Falls}_{n}$, we consider $(\sg, S)$ as a descent-starred permutation of $\{0, 1, \ldots, n\}$ that ends with a zero. We then reverse it, obtaining an ascent-starred permutation of $\{0, 1, \ldots, n\}$ that begins with a zero. We claim that tracing the inversion statistic through this process yields the identity
\begin{align*}
&\sum_{\sg \in \S_{n}} q^{\inv(\sg)} \prod_{i \in \Des(\sg) \cup \{n\}} \left( 1 + \frac{z}{q^{\inv^{\Box, i}(\sg)}} \right) \\
= &\sum_{\sg \in \S_{n}} q^{\inv(\sg)} \prod_{i \in \Asc(\sg) \cup \{0\}} \left( 1 + \frac{z}{q^{\inv{i, \Box}(\sg)}} \right) \\
= (1+z) &\sum_{\sg \in \S_{n}} q^{\inv(\sg)} \prod_{i \in \Asc(\sg)} \left( 1 + \frac{z}{q^{\inv{i, \Box}(\sg)}} \right)
\end{align*}
which indeed is $1+z$ times the corresponding term in our main theorem.

\subsection{Future Directions}
\label{ssec:final}

In this final section, we outline some possible 
generalizations and extensions of our work as well as some open questions.

First, Rawlings \cite{rawlings} 
defined a sequence of statistics on $\S_n$ called $\rmaj{r}$ defined 
for $1 \leq r \leq n$ which interpolates 
between the major index and inversion statistics. That is, for 
any $\sg \in S_n$, 
$\rmaj{1}(\sg)  = \maj(\sg)$ and $\rmaj{n}(\sg) = \inv(\sg)$.   
Rawlings went on to prove that $\rmaj{r}$ is Mahonian for any $r$. 
We plan to extend Haglund's conjecture to $\rmaj{r}$ in a future paper.

There are many other natural variations of our work from the rook theoretic point of view. In particular, various rook theory models have been developed to handle groups of colored permutations $C_{m} \wr \S_{n}$. Many mathematicians have worked to define and explore a suitable analogue of inversion number and major index to these more general groups, for example in \cite{flag}. The main 
problem here is that there is more than one way to find an 
analogue of Haglund's conjecture in this setting. This will also 
be the subject of future work.

As we mentioned in Section \ref{sec:insertion}, there is another, perhaps more well-studied, bijection between $\inv$ and $\maj$ in the $\S_{n}$ case due to Foata that is quite different from the method of Carlitz that we have used.  Is there a way to generalize Foata's insertion map to give a bijective proof of our identity? Or is Carlitz's method inherently more valuable in this setting?

If one considers $\S_{n}$ as the set of rearrangements of $\{1, 2, \ldots, n\}$, it is natural to ask what happens if we replace the underlying set $\{1, 2, \ldots, n\}$ with some other multiset. If $A$ is any multiset, 
we can replace summations over $\S_n$ in the statement of Haglund's conjecture 
by summations over all rearrangements of $A$. Haglund also conjectured 
that the resulting equality still holds. This has been 
proved by the second author \cite{multiset}. 
Such a result is especially interesting because of its connections to Macdonald polynomials and diagonal harmonics.  

The last common generalization of the symmetric group to consider is the class of reflection groups. Here we have a natural statistic that reduces to the inversion number on $\S_{n}$ (length), and even a notion of ordered set partitions, which makes the setting ideal for exploration.

\bibliographystyle{abbrv}
\bibliography{statistics}

\begin{thebibliography}{10}

\bibitem{flag}
R.~Adin and Y.~Roichman.
\newblock The flag major index and group actions on polynomial rings.
\newblock {\em EUROP. J. COMBIN}, 22:431--446, 2001.

\bibitem{carlitz}
L.~Carlitz.
\newblock A combinatorial property of $q$-{Eulerian} numbers.
\newblock {\em Amer. Math. Monthly}, 82:51--54, 1975.

\bibitem{foata}
D.~Foata.
\newblock On the {Netto} inversion number of a sequence.
\newblock {\em Proc. Amer. Math. Soc.}, 19:236--240, 1968.

\bibitem{gr-rook}
A.~M. Garsia and J.~B. Remmel.
\newblock $q$-counting rook configurations and a formula of {Frobenius}.
\newblock {\em Journal of Combinatorial Theory Series A}, 41:246--275, 1986.

\bibitem{hlr}
J.~Haglund, N.~Loehr, and J.~Remmel.
\newblock Statistics on wreath products, perfect matchings, and signed words.
\newblock {\em Eur. J. Combin.}, 26:835--868, 2005.

\bibitem{osp1}
M.~Ishikawa, A.~Kasraoui, and J.~Zeng.
\newblock {Euler-Mahonian} statistics on ordered set partitions.
\newblock {\em SIAM Journal of Discrete Mathematics}, 22:1105--1137, 2008.

\bibitem{osp2}
A.~Kasraoui and J.~Zeng.
\newblock {Euler-Mahonian} statistics on ordered set partitions (ii).
\newblock {\em Journal of Combinatorial Theory, Series A}, 116:539--563, 2009.

\bibitem{macmahon}
P.~A. MacMahon.
\newblock {\em Combinatory Analysis}, volume~1.
\newblock Cambridge University Press, 1915.

\bibitem{rawlings}
D.~Rawlings.
\newblock The $r$-major index.
\newblock {\em J. Combin. Theory Ser. A}, 115(2):175--183, 1981.

\bibitem{stein}
E.~Steingr\'{i}msson.
\newblock Statistics on ordered partitions of sets.
\newblock arXiv:math/0605670v4, April 2007.

\bibitem{ww}
M.~Wachs and D.~White.
\newblock $p, q$-{Stirling numbers} and set partition statistics.
\newblock {\em Journal of Combinatorial Theory Series A}, 56:27--46, 1991.

\bibitem{multiset}
A.~T. Wilson.
\newblock An extension of {MacMahon's} equidistribution theorem to ordered
  multiset partitions.
\newblock Extended abstract available at
  \url{sites.google.com/site/andywilsonmath/research}, 2014.

\bibitem{zeng-zhang}
J.~Zeng and C.~Zhang.
\newblock A $q$-analog of {Newton's} series, {Stirling} functions, and
  {Eulerian} functions.
\newblock {\em Results in Mathematics}, pages 370--391, 1994.

\end{thebibliography}
\label{sec:biblio}

\end{document}